\documentclass[12pt,reqno]{amsart}
\usepackage{amssymb}
\usepackage{amsmath}

\usepackage{amscd}

\newcommand{\RNum}[1]{\uppercase\expandafter{\romannumeral #1\relax}}

\usepackage[T2A]{fontenc}
\usepackage[utf8]{inputenc}
\usepackage[russian,english]{babel}
\input{int.def}

\usepackage{tikz-cd}
\usetikzlibrary{cd}
\usepackage{dirtytalk}
\usepackage{hyperref}
\usepackage{xcolor}
\usepackage{centernot}

\usepackage{amsmath}

\usepackage{enumitem}

\numberwithin{equation}{section}

\usepackage[mathcal]{euscript}

\usepackage{titlesec}
\titleformat{\section}[runin]{\bfseries}{\thesection.}{3pt}{}[.]

\myitemmargin 
\usepackage{geometry}
\newgeometry{vmargin={25mm}, hmargin={22mm,22mm}, footskip=10mm}   

\begin{document}

\title[Anti-self-dual blowups II]%
{Anti-self-dual blowups II}

\author{Vsevolod Shevchishin}
\address{Faculty of Mathematics and Computer Science\\
University of Warmia and Mazury\\ \linebreak
ul.~Słone\-czna 54, 10-710 Olsztyn, Poland
}
\email{vsevolod@matman.uwm.edu.pl, shevchishin@gmail.com}

\author{Gleb Smirnov}
\address{
Mathematical Sciences Institute, 
Australian National University, Canberra, Australia}
\email{gleb.smirnov@anu.edu.au}




\begin{abstract}
Let \(X\) be a closed, oriented four–manifold with \(b_{2}^{+} \le 3\), and suppose \(X\) contains a collection of pairwise disjoint embedded \((-2)\)–spheres. We prove that there is a Riemannian metric on \(X\) such that the Poincare dual of each of these spheres is represented by an anti–self–dual harmonic \(2\)–form. This extends our earlier result for \((-1)\)–spheres. The main new ingredient is an application of Eliashberg's \(h\)–principle for overtwisted contact structures, which we use to construct self–dual harmonic forms on four–orbifolds with prescribed local behaviour near the orbifold singular set.
\end{abstract}

\maketitle

\setcounter{section}{0}

\section{Main result}\label{main} 
Let \((X,g)\) be a closed, oriented Riemannian \(4\)-manifold, and let  
\(\Lambda^2 = \Lambda^2 T^{*}_{X}\) denote the bundle of \(2\)-forms. The Hodge operator \(\star \colon \Lambda^2 \to \Lambda^2\) has eigenvalues \(\pm 1\) and induces 
the splitting:
\[
\Lambda^2 = \Lambda^+ \oplus \Lambda^-,
\]
where \(\Lambda^\pm\) are the 
\(\pm 1\)-eigenbundles. A \(2\)-form \(\varphi\) is self-dual (SD) if \(\star \varphi = \varphi\) and anti-self-dual (ASD) if \(\star \varphi = -\varphi\). Let
\[
\mathcal{H}^2_g = \{\varphi \in \Gamma(\Lambda^2) : \Delta \varphi = 0\}
\]
be the space of harmonic \(2\)-forms. 
By Hodge theory, every harmonic \(2\)-form is closed, and the natural map 
\(\mathcal{H}^2_g \to H^2(X;\mathbb{R})\) is an 
isomorphism. 
Since \(\Delta\) commutes with \(\star\), we obtain the splitting:
\[
\mathcal{H}^2_g = \mathcal{H}_{+,g} \oplus \mathcal{H}_{-,g},
\]
where \(\mathcal{H}_{+,g}\) and \(\mathcal{H}_{-,g}\) are the spaces of harmonic SD and ASD forms. \(\mathcal{H}_{+,g},\mathcal{H}_{-,g} \subset H^2(X;\mathbb{R})\) are positive- and negative-definite with respect to the cup product, and they are orthogonal. 
Their dimensions \( b_2^\pm = \dim \mathcal{H}_{\pm,g} \) are independent of \(g\); they are the maximal dimensions of positive- and negative-definite subspaces of \(H^2(X;\mathbb{R})\), hence topological invariants of \(X\).
\smallskip%

For each metric \(g\), consider the positive-definite subspace:
\[
P(g) = \mathcal{H}_{+,g} \subset H^2(X;\mathbb{R})\quad 
\text{of dimension \(b_2^+\)}.
\]  
\(g \to P(g)\) is the period map of \(X\), introduced in \cite{LeBr-1}. \(P(g)\) depends only on the conformal class of \(g\), and the map \(P\) is open in the \(C^\infty\)-topology on metrics. 
It remains an open problem whether a given positive-definite subspace \(\pi \subset H^2(X;\mathbb{R})\) arises as \(P(g)\) for some metric \(g\).
\smallskip%

Let us now state our main result:

\begin{theorem}\label{t:main}
Let \(X\) be a closed, 
oriented \(4\)-manifold with \(b^{+}_{2} \leq 3\). 
Suppose \(X\) contains a collection of pairwise disjoint embedded spheres \(S_i\), each with self-intersection \((-2)\).  
Then there exists a Riemannian metric \(g\) on \(X\) such that every \(g\)-self-dual harmonic form \(u\) satisfies:
\[
\int_{S_i} u = 0 \quad \text{for all } S_i.
\]
Equivalently, if \(s_i \in H^2(X;\mathbb{Z})\) denotes the Poincar\'e dual of \(S_i\), then each \(s_i\) is represented by an~anti-self-dual form.
\end{theorem}
\smallskip%

Scaduto \cite{Scad} recently proved that the period map \(P\) has dense image for every \(4\)-manifold, and is 
surjective when \(b_2^+=1\). This does not imply Theorem~\ref{t:main}. Indeed, the conditions \(s_i \perp P(g)\) define a closed, positive-codimension subset of 
the Grassmannian of positive \(b_2^+\)-planes in \(H^2(X;\mathbb{R})\), and a dense open subset may avoid such a set.
\smallskip%

\begin{rem*}
Consider the following example. Let \(X\) be the underlying manifold of a smooth complex K3 surface. The cup product on \(H^2(X;\mathbb{R})\) has signature \((3,19)\). Let \(Gr^{+}_{X}\) be the Grassmannian of \(3\)-dimensional positive-definite planes in \(H^2(X;\mathbb{R})\).  For any class \(\delta \in H^2(X;\mathbb{Z})\) with \(\delta^2 = -2\), 
define:
\[
\delta^{\perp} = \{\pi \in Gr^{+}_{X} : \pi \perp \delta\}.
\]
Set:
\[
D = Gr^{+}_{X} - \bigcup_{\delta^2 = -2} \delta^{\perp}.
\]
The global Torelli theorem for K3 surfaces states that \(D\) is precisely the points in \(Gr^{+}_{X}\) realized as the period of a K{\"a}hler metric on \(X\); see \cite{Siu-2, Tod}. Each \(\delta^{\perp}\) has codimension \(2\) in \(Gr^{+}_{X}\), so \(D\) is open and dense. Although no point of $\delta^{\perp}$ can be the period of a K{\"a}hler metric, it may still occur as the period of a more general Riemannian metric. This indeed occurs: Kummer surfaces, which are K3, contain \(16\) disjoint embedded rational curves. Thus, Theorem~\ref{t:main} produces metrics for which up to \(16\) pairwise disjoint integral \((-2)\)-classes are simultaneously represented by ASD forms.
\end{rem*}
\smallskip%

In our previous work \cite{S-S-1}, we proved 
the analogue of Theorem~\ref{t:main} for collections of pairwise disjoint embedded \((-1)\)-spheres. 
Both proofs follow the same overall strategy, 
but the \((-2)\)-case presents an additional difficulty, which we now describe.
\smallskip%

The first step is to construct a Riemannian orbifold \((M,g)\) together with a closed SD form that vanishes at a prescribed point \(p \in M\). Blowing up \(M\) at \(p\) then produces a manifold \(X\) with an embedded exceptional sphere. In the \((-1)\)-case, the point \(p\) may be chosen arbitrarily, so producing a metric \(g\) with a closed SD form vanishing somewhere is relatively straightforward. In contrast, in the \((-2)\)-case, the point \(p\) must be the orbifold 
point of \(M\); otherwise the exceptional sphere created 
in the resolution has self-intersection \((-1)\) 
rather than \((-2)\). Thus, the metric \(g\) on \(M\) has to be chosen much more carefully.
\smallskip%

The new geometric input 
is the following existence theorem for closed SD forms with prescribed zeros. Its proof occupies \S\,\ref{near_contact_section}, \S\,\ref{near-symplectic-section}, and \S\,\ref{assemble-section}.
\smallskip%

\begin{proposition}\label{cone_zero}
Let \(M\) be a compact, oriented \(4\)-orbifold 
with isolated singular points \(p_1,\dots,p_n\), each of type \(\mathbb{R}^4/\{\pm1\}\), and assume \(b_2^+(M)\ge 1\). Then there exist a Riemannian metric \(g\) on \(M\), flat in neighborhoods of all \(p_i\), and a closed, \(g\)-self-dual orbifold form \(\psi\) such that \(\psi(p_i)=0\) for all \(i\).
\smallskip

For each \(p_i\), fix an orbifold chart:
\[
\pi_i : (B^4, g_{\mathrm{euc}}) \longrightarrow (U_i,g) \cong B^4/\{\pm 1\} \ni p_i,
\]
which is an isometry. An orbifold form 
means that on \(U_i - \{p_i\}\), the 
pullback \(\pi_i^*\psi\) is antipodal-invariant and extends smoothly over \(0 \in B^4\); \(\psi(p_i)=0\) means that \(\pi_i^*\psi\) vanishes at \(0 \in B^4\).
\end{proposition}
\smallskip%

The proof of 
Proposition~\ref{cone_zero} proceeds as follows. 
To begin with, we reduce the problem to constructing a near-symplectic form on \(M\) with prescribed behavior near the points \(p_i\). This is done via a theorem of Auroux-Donaldson-Katzarkov (Proposition 1 in \cite{A-D-K}). Then, following Taubes \cite{Taub-5}, we construct such a near-symplectic form via Eliashberg’s classification of overtwisted contact structures in dimension \(3\) \cite{Eliash}.
\smallskip%

With Proposition~\ref{cone_zero} in hand, 
our argument proceeds as follows. 
In \S\,\ref{neck}, we introduce an iterative neck-stretching scheme for closed SD forms on manifolds with a long cylindrical neck. In \S\,\ref{blowup}, we apply 
this scheme to the blowup of \(M\) at \(p\). 
This shows that, after inserting a sufficiently long neck, one can find a metric \(g\) for which \(P(g)\) is 
arbitrarily close to being orthogonal to a prescribed collection of \((-2)\)-classes. In \S\,\ref{section-main-proof}, we then use Proposition~\ref{cone_zero} to deform such almost-orthogonal metrics so that those \((-2)\)-classes become genuinely ASD, thereby proving Theorem~\ref{t:main}. Although these arguments closely parallel our earlier work \cite{S-S-1}, we reprove certain steps here for completeness.
\smallskip%

Combining Theorem~\ref{t:main} 
with the main result of \cite{S-S-1}, one obtains the following mixed statement.
\begin{theorem*}
Let \(X\) be a closed, 
oriented \(4\)-manifold with \(b^{+}_{2} \leq 3\). 
Suppose \(X\) contains a collection of pairwise disjoint embedded spheres \(S_i\), each with 
self-intersection \((-1)\) or \((-2)\).  
Then there exists a Riemannian metric \(g\) on \(X\) such that every \(g\)-self-dual harmonic form \(u\) satisfies:
\[
\int_{S_i} u = 0 \quad \text{for all } S_i.
\]
\end{theorem*}
\smallskip%

This statement does not require any new ideas beyond those used in the pure \((-1)\)– and \((-2)\)–cases. Our neck–stretching scheme treats all spheres simultaneously. We
therefore omit a separate proof.
\smallskip%

\begin{rem*}
A closed ASD \(2\)-form whose cohomology class is integral is the curvature of a line bundle with an ASD connection. Thus, our neck-stretching construction is formally similar to Taubes’ existence proofs for ASD connections 
\cite{Taub-6, Taub-3, Taub-4} on higher-dimensional 
bundles. However, unlike Taubes’ positive-index problems, the condition that a given set of \((-2)\)-classes be represented by ASD forms is not open in the \(C^\infty\)-topology on metrics; the index of our problem is negative.
\end{rem*}

\section{A neck-stretching argument}\label{neck} 
The gluing argument in this section is inspired by the existing literature (see, e.g., \cite{Atiyah-Pat-Sing-1, Nic-2, Taub-6}), but we present it in a form tailored to our later applications.
\smallskip%

Let $X_1$ and $X_2$ be compact, oriented $4$–manifolds with common boundary
\[
\partial X_1 = \partial X_2 = Y.
\]
Fix $\delta > 0$ small. On each $X_i$ choose a Riemannian metric $g_i$ which, on a collar neighborhood
\[
(-\delta,1] \times Y \subset X_i,
\]
splits as a product:
\begin{equation}\label{product-metric}
g_i = dt\otimes dt + g_Y.
\end{equation}
Throughout let us take \(Y = \mathbb{RP}^3\) equipped with the metric $g_Y$ induced from the unit round metric 
\(g_{S^3}\) on \(S^3\) via the double cover \(S^3 \to \mathbb{RP}^3\). All results remain valid if \(H^{1}(Y;\mathbb{R})=0\).
\smallskip%

Let
\[
P_i = [0,1]\times Y \subset X_i
\]
denote the smaller collar of the boundary.
\smallskip%

Let us form the non-compact manifold 
\(X_i^\infty\) by attaching the semi-infinite cylinder \([0,\infty)\times Y\) along \(P_i\), and extending the product metric \eqref{product-metric} over the cylindrical end. The resulting metric is denoted \(g_i^\infty\).
\smallskip%

Fix $T>0$ and consider the finite cylinder:
\[
N = [0,T]\times Y,
\]
equipped with the product metric \eqref{product-metric}. Introduce the extended cylinder:
\[
N' = [-1,\, T+1]\times Y,
\]
and attach it to \(X_1\) and \(X_2\) as follows: the collar \(P_1\) is identified with the subcylinder $[-1,0]\times Y \subset N'$, and \(P_2\) with $[T,\, T+1]\times Y \subset N'$.  For each $T$, this yields a closed $4$–manifold:
\[
X_T = X_1 \cup N \cup X_2,
\]
carrying a metric $g_T$ which equals $g_i$ on $X_i$ and the product metric on $N=[0,T]\times Y$.
\smallskip%

Inside $N$, let us single out the unit cylinders:
\[
Q_1 = [0,1]\times Y, \quad
Q_2 = [T-1,\, T]\times Y.
\]
Let
\[
\mathcal{D} = d + d^{*}, \quad
\Delta = \mathcal{D}^{\,2}
\]
denote the Hodge operator and the Hodge Laplacian.
\smallskip%

Let \(L^{2}(X_i^\infty)\) denote the space of 
$L^{2}$–integrable differential forms on \(X_i^\infty\) with respect to the metric \(g_i\). 
Let \(\mathcal{H}^{+}(X_i^\infty)\) denote 
the space of SD harmonic forms on \(X_i^\infty\) that are 
also \(L_2\), i.e.,
\[
\mathcal{H}^{+}(X_i^\infty) = 
\left\{\text{\(g_{i}^{\infty}\)-self-dual}\right\} \cap 
\operatorname{ker} \Delta \cap 
L^{2}(X_i^\infty).
\]
For a domain \(A\), use \(\|\cdot\|_{k,A}\) to denote the $L^{2}_{k}$–Sobolev norm on \(A\).
\smallskip%

Introduce cut-off functions \(\rho_1, \rho_2\) such that:
\begin{enumerate}
\item \(\rho_2 = 0\) on \(X_1 - P_1\), \(\rho_2 = 1\) on \(X_T - X_1\), and \(\rho_2\) interpolates between \(0\) and \(1\) on \(P_1\).
\smallskip

\item \(\rho_1 = 0\) on \(X_2 - P_2\), \(\rho_1 = 1\) on \(X_T - X_2\), and \(\rho_1\) interpolates between \(0\) and \(1\) on \(P_2\).
\end{enumerate}
\smallskip%

Fix a form \(\psi \in \mathcal{H}^{+}(X_1^\infty)\). Construct two sequences of smooth SD forms,
\[
v^{(i)},\ u^{(i)},\quad i \ge 1,
\]
such that:
\[
v^{(i)} \in L^{2}(X_1^\infty)\ \text{for $i$ odd},\qquad
v^{(i)} \in L^{2}(X_2^\infty)\ \text{for $i$ even},\qquad
u^{(i)} \in L^{2}(X_T)\ \text{for all $i$},
\]
via the recursion:
\smallskip%

\begin{enumerate}[label=(\alph*)]
\item \(v^{(1)}=\psi\).
\smallskip%

\item If \(i\) is odd, restrict \(v^{(i)}\) to \(X_1\cup N\cup P_2\) and set:
\[
u^{(i)}=\rho_1\,v^{(i)}.
\]
If \(i\) is even, restrict \(v^{(i)}\) to \(X_2\cup N\cup P_1\) and set:
\[
u^{(i)}=\rho_2\,v^{(i)}.
\]
\item Suppose \(i\) is odd (the even case is symmetric). Construct an SD form \(v^{(i+1)} \in L^{2}(X_2^\infty)\) 
solving
\[
\mathcal{D} v^{(i+1)}=-\mathcal{D}u^{(i)}.
\]
The construction produces elliptic estimates of the form:
\begin{equation}\label{vestim}
\|v^{(i+1)}\|_{k,\,X_2\cup Q_2}
\;\le\;
A_k\,\|\Delta u^{(i)}\|_{k-2,\,P_2},
\end{equation}
with constants \(A_k\) independent of \(T\) and \(i\). 
For even \(i\), an identical estimate holds on \(X_1 \cup Q_1\).
\end{enumerate}
\smallskip%

Let us now explain step (c) in detail for odd \(i\). The even case is similar. 
\smallskip%

To begin with, assume \(v^{(i-1)}\) and \(u^{(i-1)}\) have already been constructed. Fix numbers \(0 < T_1 < T_2\) and choose a smooth, positive function
\(\chi \colon [0,\infty)\times Y \to \mathbb{R}\) such that:
\[
\chi(t,y)=
\begin{cases}
1, & t \in [0,\,T_1],\\[4pt]
e^{-2t}, & t \ge T_2.
\end{cases}
\]
Extend \(\chi\) to all of \(X_1^\infty\) by setting 
it equal to \(1\) on \(X_1\). Consider the conformal 
rescaling:
\[
\hat g_1 = \chi\, g_1^\infty.
\]
\((X_1^\infty,\hat g_1)\) has an asymptotically Euclidean end.  
Indeed, setting \(r=e^{-t}\) yields:
\[
e^{-2t}\bigl(dt\otimes dt + g_{S^3}\bigr)
= dr\otimes dr + r^2 g_{S^3}.
\]
Thus, the region \([T_2,\infty)\times Y\) endowed with \(\hat g_1\) is isometric to the punctured closed ball of radius \(e^{-T_2}\) modulo the antipodal map. Adding the point at infinity produces a compact Riemannian orbifold, smooth away from the added point, where the metric is locally isometric to \(\mathbb{R}^4/\{\pm 1\}\).
\smallskip%

Use \(\hat X_1\) to denote the orbifold compactification of \(X_1^{\infty}\). 
\smallskip%

Choose a neighborhood \(U \subset \hat X_1\) of the orbifold point together with an isometry:
\[
B \xrightarrow{\ \pi\ } B/\{\pm 1\}
   \xrightarrow{\ \tau \ } U,
\]
where \(\pi\) identifies antipodal points and \(\tau\) is the chosen isometry. 
A differential form \(\phi\) is smooth on \(U\) if \((\pi\circ\tau)^{*}\phi\) extends smoothly across 
the origin in \(B\). A form on \(\hat X_1\) 
is smooth if it is smooth away from the orbifold point and smooth on \(U\). On the domain of such smooth forms, the operator \(\mathcal{D}\) remains elliptic and self-adjoint, so standard Hodge theory applies.
\smallskip%

\(\mathcal{D}u^{(i-1)}\) is supported in \(P_1\cup P_2\). 
Let us restrict \(\mathcal{D}u^{(i-1)}\) to \(X_1 \cup N\) and extend it by zero to the remainder of \(\hat X_1\). Let us continue to denote this extension by \(\mathcal{D}u^{(i-1)}\).
\smallskip%

Let \(\mathcal{H}(\hat X_1)\) denote the space of harmonic
\(2\)-forms on \((\hat X_1,\hat g_1)\). For any \(\varphi\in\mathcal{H}(\hat X_1)\),
\[
\int_{\hat X_1}\langle \varphi, \Delta u^{(i-1)}\rangle
=
\int_{\hat X_1}\langle \mathcal{D}\varphi, \mathcal{D}u^{(i-1)}\rangle
=0,
\]
because \(\Delta\varphi=0\) implies \(\mathcal{D}\varphi=0\). Thus, \(\Delta u^{(i-1)}\) is \(L^{2}\)–orthogonal to  
\(\mathcal{H}(\hat X_1)\). Hodge theory on the compact orbifold \((\hat X_1,\hat g_1)\) then gives a unique 
SD form \(\hat v^{(i)}\) such that:
\[
\Delta \hat v^{(i)} = -\,\Delta u^{(i-1)},\qquad  
\hat v^{(i)}\perp\mathcal{H}(\hat X_1).
\]
This can be strengthened:
\begin{lemma}
\(
\mathcal{D}\hat v^{(i)} = -\,\mathcal{D}u^{(i-1)}
\) for \(\hat v^{(i)}\) defined above.
\end{lemma}
\begin{proof}
To begin with, observe that \(d u^{(i-1)}\) is an exact \(3\)-form 
on \(\hat X_1\). Indeed:
\[
d u^{(i-1)} = d(\rho_2 v^{(i-1)}) = d\rho_2 \wedge v^{(i-1)}.
\]
Since \(H^{2}(Y;\mathbb R)=0\), there exists a \(1\)-form \(\alpha\) on \(P_1\) with \(d\alpha = v^{(i-1)}\). Hence: 
\[
d u^{(i-1)} = -d( d\rho \wedge \alpha)\quad \text{on \(P_1\).}
\]
Consider \(\hat v^{(i)} + u^{(i-1)}\). By construction:  
\[
\cald^2 (\hat v^{(i)} + u^{(i-1)}) = 0.
\]
Thus:
\[
\Delta \cald (\hat v^{(i)} + u^{(i-1)}) = 
\Delta d (\hat v^{(i)} + u^{(i-1)}) + 
\Delta d^{*} (\hat v^{(i)} + u^{(i-1)}) = 0.
\]
In particular, \(d (\hat v^{(i)} + u^{(i-1)})\) is harmonic. Since it is also exact, it must vanish. Both \(\hat v^{(i)}\) and \(u^{(i-1)}\) are self-dual; therefore:
\[
\mathcal{D}(\hat v^{(i)} + u^{(i-1)})
=
d(\hat v^{(i)} + u^{(i-1)})
-
*\,d(\hat v^{(i)} + u^{(i-1)})
= 0,
\]
and the lemma follows. \qed
\end{proof}
\smallskip%

Elliptic regularity now gives, for each \(k\),
\[
\|\hat v^{(i)}\|_{k,\hat X_1}
\;\le\;
A_k\,\|\Delta u^{(i-1)}\|_{k-2,\,\hat X_1}
=
A_k\,\|\Delta u^{(i-1)}\|_{k-2,\,P_1},
\]
with constants \(A_k\) depending only on \(\hat g_1\).
\smallskip%

Restrict \(\hat v^{(i)}\) to the cylindrical-end manifold:
\[
v^{(i)} = \hat v^{(i)}\big|_{X_1^\infty}.
\]
Choosing \(T_2\) large enough, assume \(\chi=1\) on \(X_1\cup Q_1\). On that region, the metrics \(g_1^\infty\) and \(\hat g_1\) agree, and hence:
\[
\|v^{(i)}\|_{k,\,X_1\cup Q_1}
\;\le\;
\|\hat v^{(i)}\|_{k,\hat X_1}
\;\le\;
A_k\,\|\Delta u^{(i-1)}\|_{k-2,\,P_1},
\]
which is precisely \eqref{vestim}.
\smallskip%

The $L^{2}$–norm of $2$–forms is 
conformally invariant; therefore, if  
\(\hat v^{(i)}\) is \(L^{2}\) for \(\hat g_1\), then it also \(L^{2}\) for \(g_1^\infty\).
\smallskip%

Finally, the equation
\[
\mathcal{D}v^{(i)} = -\,\mathcal{D}u^{(i-1)}
\]
holds on \((X_1^\infty,\hat g_1)\) by construction.  
To see that it also holds for \(g_1^\infty\), note that the kernel of \(\mathcal{D}\) is locally conformally invariant, so the only issue is the region where  
\(\mathcal{D}u^{(i-1)}\) is supported. But \(\mathcal{D}u^{(i-1)}\) is supported in \(P_1\), where \(\chi=1\) and \(g_1^\infty=\hat g_1\). This completes step (c).
\smallskip%

Iterating this procedure one obtains sequences \(\{u^{(i)}\}\) and \(\{v^{(i)}\}\) as above. Formally, one expects:
\[
\mathcal{D}\!\left(\sum_{i} u^{(i)}\right)=0.
\]
More precisely, for each \(i\):
\begin{itemize}
\item the support of \(\mathcal{D}u^{(i)}\) is contained in \(P_1\cup P_2\);
\smallskip

\item on \(P_2\) we have \(\mathcal{D}u^{(i)}=-\mathcal{D}u^{(i-1)}\) for even \(i\),  
while on \(P_1\) we have \(\mathcal{D}u^{(i)}=-\mathcal{D}u^{(i-1)}\) for odd \(i\).
\end{itemize}
\begin{lemma}\label{iteration}
Let \(k>0\). For all large enough \(T\), the series
\[
u=\sum_{i=1}^{\infty} u^{(i)}
\]
converges in the \(L^{2}_{k}\)-norm on \((X_T,g_T)\), and its tail satisfies:
\[
\|\,u - u^{(1)}\,\|_{k,\,X_T} = O(e^{-2T}), 
\quad
\|\,u - u^{(1)} - u^{(2)}\,\|_{k,\,X_T} = O(e^{-4T}).
\]
\end{lemma}
\begin{proof}
It is enough to find constants \(C_k>0\), independent of \(T\) and \(i\), such that:
\begin{equation}\label{geometric-step}
\|u^{(i)}\|_{k,\,X_T}
\;\le\;
C_k\,e^{-2T}\,\|u^{(i-1)}\|_{k,\,X_T}.
\end{equation}
The needed convergence then follows by summing a geometric series.
\smallskip%

To this end, let us study harmonic forms on the cylindrical neck. On \(N=[0,T]\times Y\), the Hodge Laplacian splits as:
\[
\Delta = -\partial_t^2 + \Delta_3,
\]
where \(\Delta_3\) is the Hodge Laplacian on \(Y\), acting in the \(y\)-variables only.  
Since \(\psi\) is harmonic on \(X_1^\infty\), its restriction to \(N\) admits a Fourier expansion:
\begin{equation}\label{psi-expansion}
\psi = \sum_{\lambda}
\bigl(\alpha_\lambda + \star\alpha_\lambda\bigr)\,e^{-\lambda t},
\end{equation}
where \(\lambda\) ranges over square roots (positive and negative) of eigenvalues of \(\Delta_3\) acting on \(2\)-forms, the \(\alpha_\lambda\) are \(\Delta_3\)-eigenforms with eigenvalue \(\lambda^2\), and \(\star\) is the Hodge star on \(Y\).
\smallskip%

Eigenvalues and eigenforms of \(\Delta_3\) are known (see, e.g., \cite{Foll}). $\cald\psi=0$ forces
$d\alpha_\lambda = 0$. Hence, only closed eigenforms occur.  
The smallest eigenvalue of the Hodge Laplacian on closed \(2\)-forms on the unit sphere \(S^3\) is \(4\), and these eigenforms descend to \(\mathbb{RP}^3\). Thus, the smallest eigenvalue for closed \(2\)-forms on \(Y\) is also \(4\). Consequently, every \(\lambda\) in \eqref{psi-expansion} 
satisfies \(\lambda \ge 2\). (Only the fact that the smallest eigenvalue is positive will be used below.) 
\smallskip%

The $L^2$ condition on $X_1^\infty$ rules out negative $\lambda$ in
\eqref{psi-expansion}, so we may rewrite:
\[
\psi = \sum_{\lambda\ge2}
(\alpha_\lambda+\star\alpha_\lambda)\,e^{-\lambda t}.
\]
For \(s\in[0,T]\), define the translated unit cylinder:
\[
Q_1+s=[s,s+1]\times Y\subset N.
\]
Since the \(\alpha_\lambda\) are independent of \(t\),
\[
\|\psi\|^2_{0,Q_1+s}
=
\sum_{\lambda\ge2}
\int_s^{s+1} 2\,e^{-2\lambda t}
\bigg(\int_Y\langle\alpha_\lambda,\alpha_\lambda\rangle\bigg)\,dt
\;\le\;
e^{-4s}\,\|\psi\|^2_{0,Q_1}.
\]
Thus:
\begin{equation}\label{psi-l2-decay}
\|\psi\|_{0,Q_1+s}
\;\le\;
e^{-2s}\,\|\psi\|_{0,Q_1}.
\end{equation}
Choose \(\delta>0\) small so that \(g_T\) remains product 
on the longer collars  
\([-\delta,\,1+\delta]\times Y\) and \([T-1-\delta,\,T+\delta]\times Y\) for all large \(T\). Interior elliptic estimates for the translation-invariant operator \(\Delta\) imply, for each \(k\),
\begin{equation}\label{psi-sobolev-decay}
\|\psi\|^{2}_{k,\,Q_1+s}
\le
B'_k\Bigl(
\|\psi\|^{2}_{0,\,Q_1+s+\delta}
+
\|\psi\|^{2}_{0,\,Q_1+s-\delta}
\Bigr),
\end{equation}
with a constant $B'_k$ independent of $s$ and $T$. Combining
\eqref{psi-l2-decay} and \eqref{psi-sobolev-decay}, we find: 
\[
\|\psi\|^2_{k,Q_1+s}
\le
B_k e^{-4s}\,\|\psi\|^2_{k,Q_1}
\]
for a suitable $B_k$ independent of $s$ and $T$. Summing over
$s=0,1,\dots,T-1$, we obtain:
\[
\|\psi\|^2_{k,N}
=
\sum_{j=0}^{T-1}\|\psi\|^2_{k,Q_1+j}
\le
D_k\,\|\psi\|^2_{k,Q_1},
\]
and at the far end of the neck,
\[
\|\psi\|^2_{k,P_2}
=
\|\psi\|^2_{k,Q_1+T}
\le
B_k e^{-4T}\,\|\psi\|^2_{k,Q_1}.
\]
$B_k,D_k$ depend only on $(Y,g_Y)$ and the choice of
$\delta$. Hence, the same estimates hold with \(v^{(i)}\) in place of \(\psi\):
\begin{equation}\label{v-neck}
\begin{aligned}
\|v^{(i)}\|_{k,N}^2
&\le D_k\,\|v^{(i)}\|_{k,Q_1}^2,
&
\|v^{(i)}\|_{k,P_2}^2
&\le B_k e^{-4T}\,\|v^{(i)}\|_{k,Q_1}^2,
&&\text{if $i$ is odd},\\[0.4em]
\|v^{(i)}\|_{k,N}^2
&\le D_k\,\|v^{(i)}\|_{k,Q_2}^2,
&
\|v^{(i)}\|_{k,P_1}^2
&\le B_k e^{-4T}\,\|v^{(i)}\|_{k,Q_2}^2,
&&\text{if $i$ is even}.
\end{aligned}
\end{equation}

On \(Q_1\) (for odd \(i\)) and on \(Q_2\) (for even \(i\)), we have  \(u^{(i)} = v^{(i)}\). Since 
$u^{(i)}$ and $v^{(i)}$ also agree on $N$, \eqref{v-neck} 
implies:
\begin{equation}\tag{A}\label{A}
\|u^{(i)}\|^2_{k,N}
\le
D_k\,\|u^{(i)}\|^2_{k,Q_1}
\quad\text{for $i$ odd,}
\qquad
\|u^{(i)}\|^2_{k,N}
\le
D_k\,\|u^{(i)}\|^2_{k,Q_2}
\quad\text{for $i$ even.}
\end{equation}
$\rho_1,\rho_2$ are smooth, so for we obtain:
\[
\|u^{(i)}\|^2_{k,P_2}
\le
E_k\|v^{(i)}\|^2_{k,P_2}
\quad\text{for $i$ odd,}
\qquad
\|u^{(i)}\|^2_{k,P_1}
\le
E_k\|v^{(i)}\|^2_{k,P_1}
\quad\text{for $i$ even,}
\]
Combining this
with \eqref{v-neck}, we find:
\begin{equation}\tag{B}\label{B}
\|u^{(i)}\|^2_{k,P_2}
\le
F_k e^{-4T}\,\|u^{(i)}\|^2_{k,Q_1}
\quad\text{for $i$ odd,}
\qquad
\|u^{(i)}\|^2_{k,P_1}
\le
F_k e^{-4T}\,\|u^{(i)}\|^2_{k,Q_2}
\quad\text{for $i$ even,}
\end{equation}
where $F_k = B_k E_k$.
\smallskip%

Finally, using \eqref{vestim} together with continuity of \(\Delta\), we obtain:
\begin{equation}\tag{C}\label{C}
\begin{aligned}
\|u^{(i)}\|_{k,\,X_1\cup Q_1}
&\le
A_k\,\|\Delta u^{(i-1)}\|_{k-2,\,P_1}
\le
G_k\,\|u^{(i-1)}\|_{k,\,P_1},
&&\text{if $i$ is odd},\\[0.4em]
\|u^{(i)}\|_{k,\,X_2\cup Q_2}
&\le
A_k\,\|\Delta u^{(i-1)}\|_{k-2,\,P_2}
\le
G_k\,\|u^{(i-1)}\|_{k,\,P_2},
&&\text{if $i$ is even},
\end{aligned}
\end{equation}
for some \(G_k\). 
\smallskip%

Let us now estimate \(\|u^{(i)}\|_{k,X_T}\) for odd \(i\) (the even case is symmetric). We compute:
\begin{multline*}
\|u^{(i)}\|^2_{k, X_{T}} = 
\|u^{(i)}\|^2_{k, X_1} + 
\|u^{(i)}\|^2_{k, N} + 
\|u^{(i)}\|^2_{k, P_2} 
\stackrel{\eqref{A}}{\leq} 
\|u^{(i)}\|^2_{k, X_1} + 
D_k \|u^{(i)}\|^2_{k, Q_1} + 
\|u^{(i)}\|^2_{k, P_2} 
\stackrel{\eqref{B}}{\leq}\\
\leq 
\|u^{(i)}\|^2_{k, X_1} + 
D_k \|u^{(i)}\|^2_{k, Q_1} +
e^{-4\,T} F_k \|u^{(i)}\|^2_{k, Q_1} \leq 
H_{k} \|u^{(i)}\|^2_{k, X_1 \cup Q_1} 
\stackrel{\eqref{C}}{\leq}\\ 
\leq H_{k} G_k^2 \|u^{(i-1)}\|^2_{k, P_1} 
\stackrel{\eqref{B}}{\leq} 
e^{-4\,T} F_k H_{k} G_k^2 \|u^{(i-1)}\|^2_{k, Q_2} 
\leq e^{-4\,T} C_k^2 \|u^{(i-1)}\|^2_{k, X_T}.
\end{multline*}
This completes the proof. \qed
\end{proof}
\smallskip%

Let us now isolate the contribution of \(u^{(1)}\) to the full series \(u\). Along the neck, write the Fourier expansion of \(\psi\) as:
\begin{equation}\label{psi-four}
\psi = \sum_{\lambda\ge2}\psi_\lambda,\qquad
\psi_\lambda = a_\lambda e^{-\lambda t},
\end{equation}
where each \(\psi_\lambda\) is an SD eigenform of
\(-\partial_t^2 + \Delta_3\) with parameter \(\lambda\).
\smallskip%

Let \(v^{(2)}\) be as above. Decompose it as follows:
\[
v^{(2)} = \sum_{\lambda\ge2} v^{(2)}_\lambda,
\]
where $v^{(2)}_\lambda\in L^2(X_2^\infty)$ solves
\[
\cald v^{(2)}_\lambda
=
-\cald(\rho_1\psi_\lambda).
\]
Define $u^{(2)}_\lambda = \rho_2 v^{(2)}_\lambda$ on $X_T$ so that
\[
u^{(2)} = \sum_{\lambda\ge2}u^{(2)}_\lambda.
\]
Let \(c>0\) be such that \(c^{2}\) is the next eigenvalue of \(\Delta_3\) on closed \(2\)-forms after the minimal eigenvalue \(4\). Let us show 
that \(u^{(2)}\) is well-approximated by the part 
corresponding to the lowest eigenvalue.
\begin{lemma}\label{u2}
For each \(k\), we have:
\[
\bigg\|\sum_{\lambda>2} u^{(2)}_\lambda\bigg\|_{k,X_2}
=O(e^{-cT}),
\]
for \(c > 2\) as above.
\end{lemma}
\begin{proof}
By the definition of \(u^{(2)}_\lambda\), we have:
\[
\mathcal{D}\!\left(\sum_{\lambda\ge c} u^{(2)}_\lambda\right)
=
-\,\mathcal{D}\!\left(\rho_1\sum_{\lambda\ge c}\psi_\lambda\right).
\]
Using \eqref{C}, we find:
\[
\bigg\|\sum_{\lambda\ge c}u^{(2)}_\lambda\bigg\|_{k,X_2}
\stackrel{\eqref{C}}{\leq}
G_k\Big\|\rho_1\sum_{\lambda\ge c}\psi_\lambda\Big\|_{k,P_2}
\le
G_k\sqrt{E_k}\,
\Big\|\sum_{\lambda\ge c}\psi_\lambda\Big\|_{k,P_2}.
\]
For the tail
\(
\sum_{\lambda\ge c}\psi_\lambda,
\)
repeat the decay argument used to prove \eqref{psi-l2-decay}, but using
that now $\lambda \ge c > 2$. This gives:
\[
\Big\|\sum_{\lambda\ge c}\psi_\lambda\Big\|_{k,P_2}
\le
B_k e^{-cT} \Big\|\sum_{\lambda\ge c}\psi_\lambda\Big\|_{0,Q_1}
=
O(e^{-cT}),
\]
and the lemma follows. \qed
\end{proof}

\section{Resolution formula}\label{blowup}
Let $(M,g)$ be a Riemannian $4$–orbifold, and let $p\in M$ be an isolated
orbifold point modelled on $\mathbb R^4/\{\pm1\}$ with the standard flat metric.
Thus there exists a neighbourhood $U$ of $p$ and an isometry:
\[
   U \simeq (B/\{\pm1\},\, g_0),
\]
where $B\subset\mathbb R^4$ is a ball centered at the origin, $g_0$ is the flat metric (the action of $\{\pm1\}$ is orthogonal).
\smallskip%

Choose a linear complex structure on 
\(\mathbb R^4 \cong \mathbb C^2\) such that $g_0$ is K{\"a}hler. Then view $U$ as a complex orbifold domain. Let $\omega$ denote the K{\"a}hler form on $U$ associated to $g$, and let
$r \colon M \to \mathbb R$ be the distance to $p$ with respect to 
$g$. On $U$ one has:
\[
\omega = -\, d\, d^{\mathbb C} r^{2}.
\]
After rescaling $g$, we may assume:
\[
U = \{x\in M \mid r(x) < 1\}.
\]
The unitary group $U(2)$ preserves $(\mathbb C^2, g_0)$ and commutes with the action of $\{\pm1\}$; hence, it induces an isometric action on $U$ that leaves both $r$ and $\omega$ invariant.
\smallskip%

Let \(Y \subset U\) be defined as:
\[
Y = \{x \in M \mid r(x) = 1\},
\]
and let \(g_{Y}\) be the induced metric. Geometrically, \(Y\) is the unit sphere in \(\mathbb C^{2}\) modulo the antipodal map.
\smallskip%

Choose \(a>0\) so small that \(a \cdot e < 1\) (\(\ln e = 1\)), and set:
\[
   N = \{x\in U \mid a \le r(x) \le a\cdot e\}.
\]
To apply Lemma \ref{iteration}, introduce the annuli:
\[
   P_1 = \{ x\in U \mid a\cdot e \le r(x) \le a\cdot e^{2} \},
   \qquad
   P_2 = \{ x\in U \mid a\cdot e^{-1} \le r(x) \le a \},
\]
which agree with the notation in \S\ref{neck}.
\smallskip%

The metric on \(N\) has the form:
\[
   dr \otimes dr + r^{2} g_{Y}.
\]
Hence, \(N\) is conformally equivalent to the cylinder
\([0,1] \times Y\) with the product metric:
\[
   dt \otimes dt + g_Y.
\]
Let us modify \(M\) as follows. Consider the neighbourhood of \(p\) defined as:
\[
   U_a = \{ x \in U \mid r(x) < a \} \subset M.
\]
This neighbourhood contains the cuff \(P_2\) but not the neck \(N\).  
Let
\[
   \sigma \colon X_2 \longrightarrow U_a
\]
be the minimal resolution of \(U_a\), obtained by blowing up the orbifold point \(p\). The exceptional set is an embedded sphere \(C \subset X_2\) with self–intersection \((-2)\).
\smallskip

Consider a smaller neighbourhood of \(p\) defined as:
\[
   U_\delta = \{ x \in U \mid r(x) < \delta \} \subset M.
\]
By standard complex–analytic results (see, e.g., Ch.\,1 of \cite{G-H}), the complex surface \(X_2\) admits a strictly 
plurisubharmonic function \(h\colon X_2 \to \mathbb R\) such that:
\[
   h = r^{2} \quad \text{on } X_2 - \sigma^{-1}(U_\delta).
\]
Define a K{\"a}hler form:
\[
   \omega' = -\, d\, d^{\mathbb C} h \quad \text{on } X_2,
\]
and let \(g'\) be the associated Riemannian metric. By construction, \(g\) and \(g'\) agree away from \(U_\delta\).
\smallskip%

Let us now build a manifold \(X\) by cutting out \(U_a\) from \(M\) and gluing in \((X_2, g')\) along the common boundary:
\[
   X = (M - U_a) \,\cup_{\partial X_2}\, X_2.
\]
Introduce a conformal transformation of the metric on \(X\). 
Let \(\chi \colon X \to \mathbb R\) be a positive function such that the \(\chi\)-rescaled metric on \(P_1 \cup N \cup P_2\) takes the form:
\[
   dt \otimes dt + g_Y,
\]
and such that \(\chi = 1\) away from a small neighbourhood of \(P_1 \cup N \cup P_2\). For \(T>0\), let \((X_T, g_T)\) denote the space obtained from \(X\) by replacing the factor \([0,1]\) in \(N\) by \([0,T]\).  
When \(T\) is large, \((X_T, g_T)\) contains a long cylindrical region.

\begin{lemma}\label{blowup-lemma}
Let \(M\), \(X\), and \(X_T\) be as above, and let \(C\subset X_2\) be the exceptional
\((-2)\)–sphere. Assume that \(M\) carries an SD harmonic 
form \(\psi\). Let \(\langle \psi,\omega\rangle_p\) denote the pointwise inner product with respect to \(g\), evaluated at \(p\). Then, for all large enough $T$, there exists an SD 
harmonic form $u$ on
$X_T$ such that:
\[
\|u - \psi\|_{C^0,\,X_1} \longrightarrow 0
\quad\text{as } T\to\infty,
\]
and
\begin{equation}\label{blowup-area}
\int_C u = A \,\langle \omega,\psi\rangle_p\, e^{-2T} + O(e^{-cT})
\end{equation}
for some constants $A>0$ and $c>2$ independent of $T$.
\smallskip%

To clarify the quantity \(\langle \omega,\psi\rangle_p\): choose a uniformizing chart  
\(\pi\colon (B,g_0)\to (U,g)\). Both \(\omega\) and \(\psi\) are smooth on \(U\); thus their lifts  
\(\pi^*\omega\) and \(\pi^*\psi\) are smooth, \(\{\pm 1\}\)-invariant forms on \(B\). Define \(\langle \omega,\psi\rangle_p\) by 
\(\langle \omega,\psi\rangle_p = \big\langle \pi^* \omega,\, \pi^* \psi \big\rangle_{0\in B}\), which is independent of the chosen uniformizing chart.
\end{lemma}

\begin{proof}
In the neighbourhood \(U\) of \(p\), the form \(\psi\) admits a Fourier expansion:
\[
   \psi = \sum_{\lambda \ge 2} \psi_\lambda,
   \qquad
   \psi_\lambda = a_\lambda\, r^\lambda,
\]
where each \(\psi_\lambda\) is an SD harmonic form on \(U\).
Smoothness at the orbifold point forces all exponents to be positive.
Introducing the cylindrical coordinate \(t\) via \(r=e^{-t}\), the series takes the form used in \eqref{psi-four}. 
\smallskip%

Let \(P_1, P_2\) be as above, and let \(\chi\) be the function such that the \(\chi g\) satisfies:
\[
   dt \otimes dt + g_Y
   \quad\text{on } P_1 \cup N \cup P_2,
\]
and such that \(\chi\) equals \(1\) away from a small 
neighbourhood of \(P_1 \cup N \cup P_2\).
\smallskip%

Replace the original neck \(N\) by a longer cylindrical region \([0,T]\times Y\), and analytically continue \(\psi\) from \(P_1\) to \(P_2\) along this neck. On \(P_2\), the continued 
form has the expansion:
\[
   \psi = \sum_{\lambda \ge 2} \psi_\lambda\, e^{-\lambda T},
\]
and this series converges on all of \(M_2\). For \(\lambda > 2\), the norm of \(\psi_\lambda\) tends to \(0\) as \(r\to 0\). Hence, the leading contribution to \(\langle \psi,\omega\rangle\) at \(p\) comes from the \(\lambda=2\) mode, giving:
\[
   \langle \psi,\omega\rangle_p
   = e^{-2T}\,\langle \psi_2,\omega\rangle_p.
\]
Choose holomorphic coordinates \((z,w)\) on \(\mathbb C^2\), and set:
\[
   \Omega = dz \wedge dw.
\]
The form \(\Omega\) is invariant under the action of \(\{\pm 1\}\), so it descends to \(U\). The eigenspace of the Laplacian on \(Y\) with eigenvalue \(4\) is \(3\)–dimensional, and we may write:
\[
   \psi_2 = a_1\,\omega_1 + a_2\,\omega_2 + a_3\,\omega_3,
   \qquad a_i \in \mathbb R,
\]
where:
\[
   \omega_1 = \omega,\qquad
   \omega_2 = \operatorname{Re}(\Omega),\qquad
   \omega_3 = \operatorname{Im}(\Omega).
\]
These satisfy:
\[
   \langle \omega_i,\omega_j\rangle = 2\,\delta_{ij}
   \quad\text{on } U,
\]
so in particular:
\[
   \langle \psi_2, \omega\rangle = 2\,a_1.
\]
Replacing \(U_a\) by \(X_2\) produces \(X_T\). Using Lemma \ref{iteration}, construct an SD harmonic form \(u\) on \(X_T\) such that:
\[
   u = \rho_1 \psi + u^{(2)} + O(e^{-4T}),
\]
with \(d\rho_1\) supported in \(P_2\). Using Lemma \ref{u2}, when restricted to \(X_2\), we obtain:
\[
   u^{(2)} = u^{(2)}_2 + O(e^{-cT}),
   \qquad
   \mathcal{D}u^{(2)}_2 = -\, e^{-2T}\, \mathcal{D}(\rho_1 \psi_2),
\]
for some \(c>2\). To solve
\[
   \mathcal{D} u^{(2)}_2 = -\, e^{-2T}\, \mathcal{D}(\rho_1 \psi_2),
\]
observe that \(\psi_2\) extends from \(P_2\) to a closed SD form on all of \(X_2\):
\[
   \eta
   = a_1\,\omega'
   + a_2\,\operatorname{Re}(\sigma^{*}\Omega)
   + a_3\,\operatorname{Im}(\sigma^{*}\Omega).
\]
Indeed, the holomorphic form \(\sigma^{*}\Omega\) is \(g'\)–self–dual, and \(\omega'\) is K{\"a}hler for \(g'\); both remain self–dual after introducing the conformal factor \(\chi\).
\smallskip%

Thus:
\[
   u^{(2)}_2 = e^{-2T}\,(1-\rho_1)\,\eta.
\]
Since \(\rho_1 = 0\) near the exceptional sphere \(C\), we obtain:
\[
   u = e^{-2T}\,\eta + O(e^{-cT})
   \quad\text{near } C.
\]
We compute:
\[
\int_C u
   = e^{-2T} \int_C \eta + O(e^{-cT})
   = e^{-2T} \int_C a_1\,\omega' + O(e^{-cT})
   = e^{-2T}\, \langle \psi_2,\omega\rangle_p\, A + O(e^{-cT}),
\]
for some constant \(A>0\) independent of \(T\). This completes the proof. \qed
\end{proof}

\section{Proof of Theorem \ref{t:main}}\label{section-main-proof} 
It suffices to treat the case \(b_2^{+}=3\); the 
general case is no harder. 
Let \((M,g)\) be an oriented Riemannian \(4\)-orbifold with isolated singular
points \(p_1,\ldots,p_n\), 
each of type \(\mathbb{R}^4/\{\pm1\}\). By Proposition~\ref{cone_zero}, we 
may choose \(g\) so that there exists a
\(g\)–self–dual harmonic \(2\)–form 
\(\psi_1\) which vanishes at every
\(p_\alpha\). Since \(b_2^{+}(M)=3\), we can choose further SD harmonic forms
\(\psi_2,\psi_3\) so that \(\psi_1,\psi_2,\psi_3\) 
is a basis for SD harmonic forms.
\smallskip%

For each \(\alpha = 1,\ldots,n\), fix a neighbourhood
\(V_\alpha\) of \(p_\alpha\) with an isometry
\[
V_\alpha \simeq (B/\{\pm1\}, g_0),
\]
where \(B \subset \mathbb R^4\) is a ball centered 
at the origin and
\(g_0\) is the flat metric. Choose linear complex structures \(J_\alpha\) on
\(B\) such that \(g_0\) is K{\"a}hler. \(J_\alpha\) 
is invariant under the
antipodal map, so it descends to \(V_\alpha\). Let \(\omega_\alpha\) be 
the resulting K{\"a}hler form on \(V_\alpha\).
\smallskip%

By rotating \(J_\alpha\), we
may arrange that, for each \(\alpha\),
\[
\langle \omega_\alpha , \psi_1 \rangle_{p_\alpha}
=
\langle \omega_\alpha , \psi_2 \rangle_{p_\alpha}
= 0.
\]
For each \(\alpha\), let 
\(D_{\alpha} \subset \mathbb R^3\) be 
a \(3\)–disc with coordinates 
\((s_{\alpha, 1}, s_{\alpha, 2}, s_{\alpha, 3})\). Set: 
\[
D^{3n} = D_1 \times \cdots \times D_n \subset \mathbb R^{3n}.
\]
Let \(g_s\), \(s\in D^{3n}\), be a family of orbifold metrics with \(g_0 = g\). 
Let \(a_i \in H^2(M; \mathbb R)\) denote the cohomology class of \(\psi_i\), 
and let \(\psi_{i,s}\) be the \(g_s\)–self–dual part of the \(g_s\)–harmonic
representative of \(a_i\). For each \(\alpha\) 
and \(i\), define:
\[
\pi_{\alpha, i}(s)
=
\big\langle \omega_\alpha , \psi_{i,s} \big\rangle_{p_\alpha},
\quad
\alpha=1,\ldots,n,\ i=1,2,3,
\]
where the inner product is computed using \(g_s\).
By construction, \(\pi_{\alpha, i}(0)=0\) 
for all \(\alpha,i\).
\begin{lemma}\label{brun}
There exist neighbourhoods
\(U_\alpha \subset V_\alpha\) of \(p_\alpha\) and a family of metrics
\(g_s\), \(s\in D^{3n}\), such that
\(g_s = g\) on each \(U_{\alpha}\) for all \(s \in D^{3n}\), and
such that the map
\[
s \to \big(\pi_{\alpha, i}(s)\big)_{\alpha, i},
\]
is a local diffeomorphism near \(s=0\).
\end{lemma}
\begin{proof}
The proof is identical to that of Lemma~4 in \cite{S-S-1}, which relies on \cite{LeBr, Honda}. \qed
\end{proof}
\smallskip%

For simplicity, assume \(n = 1\) and drop 
index \(\alpha\); the general case is no harder. Let \((M, g_s)\) be as in Lemma \ref{brun}. Let us apply the neck–stretching procedure of \S\ref{blowup} inside \(U\). This produces a family \((M, g_{T,s})\) which agrees with \(g_s\) on \(U\) (technically, on a slightly smaller 
neighbourhood). Let us next resolve the double point \(p\) using the complex structure \(J\). Let \(C\) denote the exceptional \((-2)\)-sphere, and let the resulting smooth 
family of \(4\)-manifolds be denoted by \(X_{T,s}\). 
\smallskip%

By Lemma~\ref{blowup-lemma}, there exist SD harmonic forms \(u_{i,s}\) on \(X_{T,s}\) such that:
\begin{equation}\label{eq:ui-orb}
\int_C u_{i,s}
=
A\,\pi_i(s)\,e^{-2T} + O(e^{-cT}),
\quad A > 0,\ c>2.
\end{equation}
For large \(T\), the error term in \eqref{eq:ui-orb} is dominated by the leading contribution; hence, the map
\[
s \to 
\left(
\int_C u_{1,s},\,
\int_C u_{2,s},\,
\int_C u_{3,s}
\right)
\]
is surjective onto a neighbourhood of the origin in \(\mathbb{R}^3\). This completes the proof.

\section{Near-contact forms}\label{near_contact_section}
In this section we consider contact forms with isolated, nondegenerate singularities on \(3\)-manifolds. The goal is to show how Eliashberg's \(h\)-principle gives a complete classification of such forms up to deformation.
\smallskip%

Let \(Y\) be a \(3\)-manifold. 
A \(1\)-form \(\lambda\) on \(Y\) is called a \emph{near-contact form} if:
\begin{enumerate}
\item \(\lambda\) is transverse to the zero section of \(T^{*}_Y\), so its zero set 
\(\lambda^{-1}(0)=\{p_1,\ldots,p_k\}\) is a discrete (for simplicity, finite) set of points;
\smallskip%

\item there exists a volume form \(\operatorname{vol}\) on \(Y\) such that, writing
\[
\lambda \wedge d\lambda = f\,\operatorname{vol}.
\]
the function \(f\) is 
positive on \(Y - \{p_1,\ldots,p_k\}\), and each \(p_i\) is a nondegenerate local minimum of \(f\). (This condition then holds for any volume form inducing the same orientation as \(\lambda \wedge d\lambda\).)
\end{enumerate}

\begin{lemma}\label{lambda_closed}
Suppose that \(p\) is a transverse zero of \(\lambda\) and that \(\lambda\) is contact in a small punctured neighbourhood of \(p\); then \(d\lambda\) also vanishes at \(p\).
\end{lemma}
\begin{proof}
Choose a connection \(\nabla\) near \(p\). Since \(\lambda\) is contact in a punctured neighbourhood of \(p\), we have:
\[
\nabla(\lambda \wedge d\lambda)\big|_p = 0.
\]
Using
\[
\nabla(\lambda \wedge d\lambda)
= \nabla\lambda \wedge d\lambda + \lambda \wedge \nabla d\lambda,
\]
we obtain:
\[
\nabla\lambda|_p \wedge d\lambda|_p = 0.
\]
The transverse vanishing of \(\lambda\) at \(p\) now implies 
\(d\lambda|_p = 0\). This completes the proof. \qed
\end{proof}
\smallskip%

If \(p\) is a transverse zero of a \(1\)-form \(\lambda\) and 
\(d\lambda\) also vanishes at \(p\), then the bilinear form:
\begin{equation}\label{form_A}
A(\xi,\eta) = \nabla_{\xi}\lambda(\eta) \quad 
\text{for }\xi,\eta\in T_Y|_{p},
\end{equation}
is symmetric and nondegenerate. Since \(\lambda|_p=0\), the form \(A\) does not depend on \(\nabla\).
\smallskip%

\begin{lemma}
If a near-contact \(1\)-form \(\lambda\) vanishes at \(p\), then 
the form \(A\) defined by \eqref{form_A} is not definite.
\end{lemma}
\begin{proof}
Choose a flat metric \(g\) near \(p\), let \(\nabla\) be its flat connection, 
and let \(\operatorname{vol}_g\) be its volume form. 
Let \(Y\) be the vector field dual to \(d\lambda\) with respect to \(\operatorname{vol}_g\); that is, 
\[
\iota_Y \operatorname{vol}_g = d\lambda.
\]
Then:
\[
L_Y \operatorname{vol}_g = (\operatorname{div} Y)\,\operatorname{vol}_g = 0.
\]
Thus, the endomorphism field \(\nabla Y\colon T_Y\to T_Y\) is traceless everywhere. In particular, the linear map:
\[
B(\xi)=\nabla_\xi Y \in T_Y|_p \quad \text{for all }\xi\in T_Y|_p,
\]
satisfies \(\operatorname{tr} B = 0\). We compute:
\[
\lambda \wedge d\lambda = \lambda(Y)\, \operatorname{vol}_g.
\]
Since \(\lambda\) and \(Y\) both vanish at \(p\), it 
follows that:
\[
\nabla_\xi \nabla_\eta \lambda(Y)\big|_p
= A(\xi, B\eta) + A(\eta, B\xi)
\quad \text{for all }\xi,\eta\in T_Y|_p.
\]
Assume for contradiction that \(A\) is definite. 
Choose \(g\) so that \(g = A\) at \(p\). Set:
\[
H(\xi,\xi)=\nabla_\xi^2 \lambda(Y)\big|_p 
= 2 g(\xi, B\xi).
\]
For any \(g\)-orthonormal basis \(\{e_1,e_2,e_3\}\), we obtain:
\[
\operatorname{tr}_g H 
= \sum_{i=1}^3 H(e_i,e_i)
= 2\,\operatorname{tr} B = 0.
\]
Hence, the function \(f=\lambda(Y)\) cannot have a nondegenerate maximum or minimum at \(p\). This contradiction completes the proof. \qed  
\end{proof}
\smallskip%

Define the index of a transverse zero \(p\) of \(\lambda\) as:
\[
\operatorname{ind}_{p}\lambda = 
\operatorname{sgn}\operatorname{det} A.
\]
For a nondegenerate zero \(p\) of a near-contact form \(\lambda\), this index provides a topological obstruction to deforming one such form into another (through near-contact forms with a nondegenerate zero at \(p\)). 
Another obstruction is the orientation induced by \(\lambda\wedge d\lambda\) on a punctured neighbourhood of \(p\). We now show that there are no further obstructions, even for compactly supported deformations.
\begin{lemma}\label{local_lambda}
Let \(\lambda\) and \(\lambda'\) be near-contact \(1\)-forms defined on a neighbourhood of a point \(p \in Y\). 
Assume that \(\lambda\) and \(\lambda'\) both vanish at \(p\), that
\begin{equation}\label{ind_equal}
\operatorname{ind}_{p}(\lambda)=\operatorname{ind}_{p}(\lambda'),
\end{equation}
and that \(\lambda\wedge d\lambda\) and \(\lambda'\wedge d\lambda'\) induce the same orientation on a punctured neighbourhood of \(p\). 
Let \(W\) be any neighbourhood of \(p\) on which both forms are defined. 
Then there exist nested neighbourhoods
\[
p \in U \subset V \subset W
\]
and a family of near-contact forms \(\lambda_t,\ t\in[0,1],\) such that:

\begin{enumerate}
\item \(\lambda_0 = \lambda\) on \(W\),  
\(\lambda_t = \lambda\) on \(W - V\) for all \(t\in[0,1]\),  
and \(\lambda_1 = \lambda'\) on \(U\);
\smallskip

\item for each \(t\in[0,1]\), the form \(\lambda_t\) is near-contact on \(V\), and \(p\) is the unique zero of \(\lambda_t\) in \(V\).
\end{enumerate}
\end{lemma}
\begin{proof}
To begin with, 
identify the neighbourhood \(W\) of \(p\) with the 
unit ball \(B_1\subset\mathbb{R}^3\), and let \(p\) 
correspond to the origin. Choose a flat metric 
\(g\) in \(\mathbb R^3\), let \(\nabla\) be its flat 
connection, let \(\operatorname{vol}_g\) be its volume form, and 
let \(r \colon \mathbb R^3 \to R\) be the distance from the origin.
\smallskip%

Using \eqref{ind_equal}, we may assume that:
\[
\nabla\lambda|_p = \nabla\lambda'|_p.
\]
Then:
\[
\nabla^2\!\left(
\big[(1-t)\lambda + t\lambda'\big] \wedge
d\big[(1-t)\lambda + t\lambda'\big]
\right)\!\big|_p
=
(1-t)\,\nabla^2(\lambda\wedge d\lambda)\big|_p
+
t\,\nabla^2(\lambda'\wedge d\lambda')\big|_p.
\]
Let\(f\) and \(f'\) be defined as:
\[
\lambda \wedge d \lambda = f \operatorname{vol}_g, \quad 
\lambda' \wedge d \lambda' = f' \operatorname{vol}_g.
\]
Then:
\[
(1-t)\,\nabla^2(\lambda\wedge d\lambda)\big|_p
+
t\,\nabla^2(\lambda'\wedge d\lambda')\big|_p
=
\big[(1-t)\nabla^2 f + t\nabla^2 f'\big]\,\operatorname{vol}_g\big|_p.
\]
Since \(\lambda\wedge d\lambda\) and \(\lambda'\wedge d\lambda'\) induce the same orientation near \(p\), let us choose that local orientation so that both \(f\) and \(f'\) are nonnegative. Each has a nondegenerate minimum at \(p\), hence their convex combination does also. Thus, there exists \(c>0\) such that for all \(t\):
\begin{equation}\label{uniform_morse}
(1-t)f + t f' \ge c r^2
\quad\text{near }p.
\end{equation}
Choose \(0< \varepsilon, R <1\). 
Let \(B_R\) and \(B_{\varepsilon R}\) be 
the balls of radii \(R\) and \(\varepsilon R\). 
Let \(\rho=\rho(r)\) be a nonincreasing cut-off function such that:
\[
\rho|_{B_{\varepsilon R}} = 1,
\qquad
\rho|_{B_1 - B_R} = 0.
\]
Define:
\[
\lambda_t
=
\lambda + t\,\rho\,(\lambda' - \lambda).
\]
For \(W = B_1\), \(V = B_R\), and \(U = B_{\varepsilon R}\), property (1) 
of the lemma is immediate. Let us show \(\varepsilon\) and \(R\) can be chosen so that (2) holds. We compute:
\begin{align*}
\lambda_t \wedge d\lambda_t
&=
\lambda\wedge d\lambda
+ t\rho(\lambda'-\lambda)\wedge d\lambda
+ t\,\lambda\wedge d\!\big(\rho(\lambda'-\lambda)\big)
+ t^2\rho(\lambda'-\lambda)\wedge d\!\big(\rho(\lambda'-\lambda)\big) = 
\\[3pt]
&=
(1-t\rho)\,\lambda\wedge d\lambda
+ t\rho\,\lambda'\wedge d\lambda'
- t\rho(\lambda'-\lambda)\wedge d(\lambda'-\lambda)
\\[0pt]
& \qquad\qquad\qquad
- t\,d\rho\wedge\lambda\wedge(\lambda'-\lambda)
+ t^2\rho^2(\lambda'-\lambda)\wedge d(\lambda'-\lambda) = 
\\[4pt]
&=
\big[(1-\rho t)f + \rho t f'\big]\operatorname{vol}_g
- t\,d\rho \wedge \lambda\wedge(\lambda'-\lambda)
- t\rho(1-t\rho)\,(\lambda'-\lambda)\wedge d(\lambda'-\lambda).
\end{align*}
Using \(0 \le \rho \le1\) and \eqref{uniform_morse}, 
choose \(R > 0\) small enough so that:
\begin{equation}\label{cr^2}
(1-\rho t)f + \rho t f' \ge c r^2
\quad\text{on }B_R.
\end{equation}
For a domain \(A\), write \(\|\cdot\|_{k,A}\) for the \(C^k\)-norm on \(A\).
\smallskip%

Using:
\[
\lambda|_p=\lambda'|_p = 0,
\quad
\nabla(\lambda'-\lambda)|_p = 0,
\quad
d\lambda|_p = d\lambda'|_p = 0,
\]
there exists \(C>0\) such that for all \(r\),
\begin{equation}\label{lambda_ck}
\|\lambda\|_{0,B_r} \le Cr,\quad
\|\lambda'\|_{0,B_r} \le Cr,\quad
\|\lambda'-\lambda\|_{0,B_r} \le Cr^2,
\quad
\|d\lambda\|_{0,B_r} \le Cr,\quad
\|d\lambda'\|_{0,B_r} \le Cr.
\end{equation}
Let \(f_t\) be defined as:
\[
\lambda_t\wedge d\lambda_t = f_t\,\operatorname{vol}_g.
\]
On \(B_R\):
\begin{align*}
f_t
&\ge
\big[(1-\rho t)f + \rho t f'\big]
 - t\|d\rho\|_{0,B_r}\|\lambda\|_{0,B_r}\|\lambda'-\lambda\|_{0,B_r}
 - t\rho(1-t\rho)\|\lambda'-\lambda\|_{0,B_r}\|d(\lambda'-\lambda)\|_{0,B_r} \ge \\ 
&\stackrel{\eqref{cr^2}}{\ge}
c r^2
 -
\|d\rho\|_{0,B_r}\|\lambda\|_{0,B_r}\|\lambda'-\lambda\|_{0,B_r}
 -
\|\lambda'-\lambda\|_{0,B_r}\|d(\lambda'-\lambda)\|_{0,B_r} \ge \\
& \stackrel{\eqref{lambda_ck}}{\ge}
c r^2 - C^2 r^3\|d\rho\|_{0,B_r} - 2C^2 r^3.
\end{align*}
Fix \(R\). 
For each \(\delta>0\), there exists \(\varepsilon>0\), depending on \(\delta\) but not on \(R\), and a nonincreasing 
\(\rho=\rho(r)\) such that:
\[
\rho=1 \text{ on } r\le\varepsilon R,
\qquad
\rho=0 \text{ on } r\ge R,
\qquad
\partial_r\rho \ge -\delta/r.
\]
In the logarithmic coordinate \(t=\ln r\), this is immediate: 
we want \(\rho=1\) for \(t\le \ln R - \ln(1/\varepsilon)\) and \(\rho=0\) for \(t\ge\ln R\), with \(\partial_t\rho\ge -\delta\).
\smallskip%

With such a choice:
\[
f_t \ge c r^2 - (c/2) r^2 - 2C^2 r^3
\quad\text{on }B_R.
\]
Thus, for \(R>0\) small enough, \(f_t \ge \frac{c}{3} r^2\) on \(B_R\). This completes the proof. \qed
\end{proof}
\smallskip%

Taubes \cite{Taub-5} constructs a near-contact form on \(\mathbb{R}^3\) with a zero at \(0\in\mathbb{R}^3\) such 
that every sufficiently small punctured neighbourhood of \(0\) contains an overtwisted disc. 
For completeness, we give a proof of this result by 
constructing a sequence of such forms.

\begin{lemma}\label{overtwisted}
There exists a near-contact form \(\lambda\) defined in a neighbourhood of \(0\in\mathbb{R}^3\) such that \(\lambda\) has a zero at \(0\), and such that for each small enough neighbourhood \(0\in U\subset\mathbb{R}^3\), the punctured set 
\(U - \{0\}\) contains an overtwisted disc.
\end{lemma}
\begin{proof}
Choose a flat metric 
\(g\) in \(\mathbb R^3\), let \(\nabla\) be its flat 
connection, and let \(\operatorname{vol}_g\) be its volume 
form.
\smallskip%

Fix any near-contact form \(\lambda\) 
defined in a neighbourhood of \(0\in\mathbb{R}^3\) such that \(\lambda\) has a zero at \(0\) and no other zeros. Assume that the orientations induced by \(\operatorname{vol}_{g}\) and \(\lambda\wedge d\lambda\) agree.
\smallskip%

Consider the sphere 
\(S_\varepsilon \subset \mathbb{R}^3\) of radius \(\varepsilon\). 
We shall show that there exists a homogeneous cubic 
polynomial \(C\) such that, 
for each small enough \(\varepsilon\), the \(1\)-form
\[
\lambda + dC,
\]
when restricted to \(S_\varepsilon\), 
contains an overtwisted disc. Since \(C\) is homogeneous cubic and
\[
(\lambda + dC)\wedge d(\lambda + dC)
=
\lambda\wedge d\lambda + dC\wedge d\lambda,
\]
the form \(\lambda + dC\) remains near-contact in a 
neighbourhood of the origin.
\smallskip%

To proceed, write \(\lambda\) as:
\[
\lambda = da + \mu,
\]
where \(a\colon\mathbb{R}^3\to\mathbb{R}\) is given by
\[
a(x) = \tfrac{1}{2}\langle Ax, x\rangle,\quad x\in\mathbb{R}^3,
\]
with \(A\) defined as in \eqref{form_A}. By changing the sign of \(\lambda\) if needed, we may assume that in some orthonormal frame the quadratic form \(a\) is given by
\[
a(x) = \tfrac{1}{2}(x_1^2 + x_2^2 - x_3^2).
\]
Define the dilation \(\varphi_\varepsilon\colon\mathbb{R}^3\to\mathbb{R}^3\) as:
\[
\varphi_\varepsilon(x) = \varepsilon x,\quad x\in\mathbb{R}^3.
\]
Consider the unit sphere \(S_1\subset\mathbb{R}^3\), and 
let \(V\) be a small neighbourhood of \(S_1\). 
Since both \(\mu\) and \(\nabla\mu\) vanish at the origin \(0\in\mathbb{R}^3\), there exists a constant \(D>0\), depending only on \(\mu\), such that for each \(\varepsilon\),
\begin{equation}\label{mu_eps_norm}
\|\varphi_\varepsilon^{*}\mu\|_{1,V} \le D\,\varepsilon^{3}.
\end{equation}
We compute:
\[
\varphi_{\varepsilon}^{*}(\lambda + dC)
=
\varepsilon^{2}\left(
da + \frac{\varphi_{\varepsilon}^{*}\mu}{\varepsilon^{2}}
+ \varepsilon\, dC
\right).
\]
In particular, the 
\(\varepsilon^{-2}\,\varphi_{\varepsilon}^{*}(\lambda + dC)\) 
is an \(\varepsilon\)-small \(C^{1}\)-perturbation of \(da\). 
We now choose this perturbation in a controlled way. 
\smallskip%

Restrict \(a\) to the unit sphere \(S_{1}\). 
The critical set of \(a|_{S_{1}}\) is as follows:
\begin{enumerate}
\item The two points \(x_{3} = \pm 1\), which are nondegenerate critical points.
\smallskip%

\item The circle \(x_{3} = 0\), which consists entirely of critical points. This circle is Morse-Bott for \(a|_{S_{1}}\). Indeed:
\[
\nabla_{x_{3}} da = -dx_{3}.
\]
\end{enumerate}
\smallskip%

Choose \(C\) so that the \(2\)-form
\[
dC \wedge \nabla_{x_3} da = -\,dC \wedge dx_3,
\]
when restricted to the unit sphere \(S_{1}\), 
does not vanish at the points of the big circle \(x_3 = 0\). It suffices to take \(C = C(x_1,x_2)\), independent of \(x_3\), and arranged so that \(dC \neq 0\) unless \(x_1 = x_2 = 0\).
\smallskip%

Since
\[
\varphi_{\varepsilon}^{*}(\mu + dC)\wedge 
\nabla_{x_3}\varphi_{\varepsilon}^{*} da
=
\varepsilon^{2}\,\varphi_{\varepsilon}^{*}\mu \wedge \nabla_{x_3} da
\;+\;
\varepsilon^{5}\, dC \wedge \nabla_{x_3} da,
\]
it follows from \eqref{mu_eps_norm} that, after multiplying \(C\) by a sufficiently large constant if needed, the \(2\)-form
\[
\varphi_{\varepsilon}^{*}(\mu + dC)\wedge 
\nabla_{x_3}\varphi_{\varepsilon}^{*} da,
\]
when restricted to \(S_{1}\), does not vanish at any point of the circle \(x_3 = 0\).
\smallskip%

With this choice of \(C\), and for \(\varepsilon\) small enough, the \(1\)-form
\[
\varepsilon^{-2}\,\varphi_{\varepsilon}^{*}(\lambda + dC)
=
da \;+\; \frac{\varphi_{\varepsilon}^{*}\mu}{\varepsilon^{2}} \;+\; \varepsilon\, dC,
\]
when restricted to \(S_{1}\), has exactly two transverse zeros \(p_{+}\) and \(p_{-}\). These are small perturbations of the points \(x_3 = 1\) and \(x_3 = -1\), respectively.
\smallskip%

Set:
\[
\lambda_{\varepsilon}
=
da
+ \frac{\varphi_{\varepsilon}^{*}\mu}{\varepsilon^{2}}
+ \varepsilon\, dC.
\]
We will show that for each sufficiently small but nonzero \(\varepsilon\), the restriction \(\lambda_{\varepsilon}|_{S_{1}}\) admits an overtwisted disc. Consider the Euler field:
\[
E = x_{1}\partial_{x_{1}} + x_{2}\partial_{x_{2}} + x_{3}\partial_{x_{3}},
\]
and define \(\omega\) by \(\iota_{E}\operatorname{vol}_{g} = \omega\). Its restriction to \(S_{1}\) is nondegenerate. 
Restrict \(\omega\) and 
\(\lambda_{\varepsilon} \) to \(S_{1}\), and let \(X_{\varepsilon}\) be the \(\omega\)-dual of \(\lambda_{\varepsilon}|_{S_{1}}\); 
that is, \(\iota_X \omega = \lambda_{\varepsilon}|_{S_{1}}\). 
Since \(\lambda_{\varepsilon}\) is \(C^{1}\)-close to \(da\), the vector field \(X_{\varepsilon}\) is \(C^{1}\)-close to \(X_{0}\).
\smallskip%

\(X_{\varepsilon}\) has exactly two zeros \(p_{+}\) and \(p_{-}\) on \(S_{1}\), both transverse. 
Because \(a\) has nondegenerate minima at \(x_{3} = \pm 1\), the linearization of \(X_{0}\) at these points has a pair of imaginary eigenvalues; therefore, the linearization of \(X_{\varepsilon}\) at \(p_{\pm}\) cannot have purely real eigenvalues, so neither \(p_{+}\) nor \(p_{-}\) is a saddle. We will show that:
\begin{equation}\label{divX}
\operatorname{div}X_{\varepsilon}(p_{+})<0,
\qquad
\operatorname{div}X_{\varepsilon}(p_{-})<0
\quad
\text{for all small enough but nonzero \(\varepsilon\).}
\end{equation}
Let \(f_{\varepsilon}\) be defined by:
\[
\lambda_{\varepsilon}\wedge d\lambda_{\varepsilon}
=
f_{\varepsilon}\,\operatorname{vol}_{g}.
\]
We compute:
\[
f_{\varepsilon}\,\omega
=
f_{\varepsilon}\, \iota_{E}\operatorname{vol}_{g}
=
\iota_{E}(\lambda_{\varepsilon}\wedge d\lambda_{\varepsilon})
=
\lambda_{\varepsilon}(E)\, d\lambda_{\varepsilon}
- \lambda_{\varepsilon}\wedge\iota_{E} d\lambda_{\varepsilon}.
\]
Restricting to \(S_{1}\):
\[
f_{\varepsilon}\,\omega
=
\lambda_{\varepsilon}(E)\,(\operatorname{div}X)\,\omega
- \lambda_{\varepsilon}\wedge\iota_{E} d\lambda_{\varepsilon}.
\]
Restricting to \(p_{+}\) 
and \(p_{-}\) (where \(\lambda_{\varepsilon} = 0\)):
\[
f_{\varepsilon}(p_{+})
=
\lambda_{\varepsilon}(E)(p_{+})\,\operatorname{div}X_{\varepsilon}(p_{+}),
\qquad
f_{\varepsilon}(p_{-})
=
\lambda_{\varepsilon}(E)(p_{-})\,\operatorname{div}X_{\varepsilon}(p_{-}).
\]
Since \(f_{\varepsilon}>0\) everywhere on \(S_{1}\) and
\[
\lambda_{\varepsilon}(E)(p_{+}) \to -1,
\qquad
\lambda_{\varepsilon}(E)(p_{-}) \to -1
\quad\text{as }\varepsilon \to 0,
\]
we obtain \eqref{divX}. Consequently, 
both \(p_{+}\) and \(p_{-}\) are sinks of \(X_{\varepsilon}\), and \(X_{\varepsilon}\) has no other zeros on \(S_{1}\). Hence, \(X_{\varepsilon}\) must admit a limiting cycle. This completes the proof. \qed
\end{proof}
\smallskip%

Recall the following theorem of Eliashberg (Theorem 3.1.1 in \cite{Eliash}).
\begin{theorem*}[\cite{Eliash}]
Let \(M\) be a compact \(3\)-manifold 
and let \(A\subset M\) be a closed subset 
such that \(M - A\) is connected. Let \(K\) be a compact space 
and \(L\subset K\) a closed subspace. 
Let \(\{\xi_{t}\}_{t\in K}\) be a family of 
\(2\)-plane distributions which are contact for all \(t\in L\) and are contact near \(A\) for all \(t\in K\). 
Assume there exists an embedded disc \(\Delta \subset M - A\) 
such that \(\xi_{t}\) is contact near \(\Delta\) and \((\Delta,\xi_{t})\) is an overtwisted disc for all \(t\in K\). 
Then there exists a family of contact structures \(\{\xi'_{t}\}_{t\in K}\) on \(M\) such that \(\xi'_{t}\) agrees with \(\xi_{t}\) near \(A\) for all \(t\in K\), and agrees with \(\xi_{t}\) everywhere for all \(t\in L\). 
\end{theorem*}
\smallskip%

\begin{proposition}\label{contact_classification}
Let \(Y\) be a closed \(3\)-manifold, and let \(\lambda_{0}\) and \(\lambda_{1}\) be near-contact forms on \(Y\) 
with the same zero set:
\[
\lambda_{0}^{-1}(0)=\lambda_{1}^{-1}(0)
=
\{p_{1},\ldots,p_{k}\},\quad k>0.
\]
Suppose \(\lambda_{0}\) and \(\lambda_{1}\) are 
homotopic as nonzero \(1\)-forms on \(Y - \{p_{1},\ldots,p_{k}\}\), and that \(\lambda_{0}\wedge d\lambda_{0}\) and \(\lambda_{1}\wedge d\lambda_{1}\) induce the same orientation on \(Y\). Then there exists a near-contact homotopy \(\lambda_{t}\), \(t\in[0,1]\), between \(\lambda_{0}\) and \(\lambda_{1}\) on the whole of \(Y\).
\end{proposition}
\begin{proof}
Since \(\lambda_{0}\) and \(\lambda_{1}\) are homotopic as nonvanishing \(1\)-forms on \(Y - \{p_{1},\ldots,p_{k}\}\), we have for each zero \(p_{i}\):
\[
\operatorname{ind}_{p_{i}}(\lambda_{0})
=
\operatorname{ind}_{p_{i}}(\lambda_{1}).
\]
Using Lemmas \ref{overtwisted} and \ref{local_lambda}, we may assume that there exist neighbourhoods \(U_{i}\ni p_{i}\) such that  
\(\lambda_{0}|_{U_{i}}=\lambda_{1}|_{U_{i}}\), and that \((U_{i},\lambda_{0})\) contains an overtwisted disc. 
Choose closed neighbourhoods \(A_{i}\subset U_{i}\) of \(p_{i}\) such that \(U_{i} - A_{i}\) still contains an overtwisted disc. 
Set \(A=A_{1}\cup\cdots\cup A_{k}\); shrinking the \(A_{i}\) if necessary, we may assume \(Y -  A\) is connected. We can now construct a family of \(1\)-forms \(\alpha_{t}\), \(t\in[0,1]\), such that:
\begin{enumerate}
\item \(\alpha_{0}=\lambda_{0}\), \(\alpha_{1}=\lambda_{1}\), and  
\(\alpha_{t}=\lambda_{0}=\lambda_{1}\) on each \(U_{i}\);
\smallskip
\item for all \(t\), the form \(\alpha_{t}\) is nowhere zero 
on \(Y - \{p_{1},\ldots,p_{k}\}\).
\end{enumerate}
\smallskip%

The lemma now follows from Eliashberg's theorem. \qed
\end{proof}

\section{Near-symplectic forms}\label{near-symplectic-section}
In this section we review the 
notion of a near-symplectic form; see 
\cite{A-D-K, Taub-5, Perutz, Lekili, Gay-Kirby} for background. Our goal is to prove Proposition \ref{near-ball}.
\smallskip%

Let \(X\) be a \(4\)-manifold. 
A \(2\)-form \(\omega\) on \(X\) is called 
a \emph{near-symplectic form} if:
\begin{enumerate}
\item \(d\omega = 0\);
\smallskip

\item \(\omega\) is symplectic away 
from its zero set, i.e., for each 
\(x\in X\) with \(\omega_x \neq 0\),
\(\omega_x\wedge\omega_x \neq 0\);
\smallskip

\item the zero set 
\(Z(\omega)=\{x\in X \mid \omega_x=0\}\) 
is a smooth \(1\)-dimensional submanifold (a union of embedded circles and arcs), and for each \(x\in Z\) 
and \(\xi \in T_X|_x\),
\[
(\nabla_\xi\omega)_x\wedge(\nabla_\xi\omega)_x\neq 0.
\]
If \(X\) has boundary, \(Z\) is additionally required to be transverse to \(\partial X\).
\end{enumerate}
\smallskip%

This definition is equivalent to the following. 
Let \(\omega\) be a closed \(2\)-form with zero set \(Z=\omega^{-1}(0)\). Then \(\omega\) is near-symplectic iff \(Z\) is a smooth \(1\)-manifold and there exists a volume form \(\operatorname{vol}\) on \(X\) such that:
\[
\omega\wedge\omega=f\,\mathrm{vol},
\]
where \(f>0\) on \(X-Z\) and \(Z\) is a Morse-Bott critical submanifold of \(f\); that is, \(df=0\) along \(Z\) and the Hessian of \(f\) is nondegenerate on the normal bundle of \(Z\).
\smallskip%

Endow \(X\) with a Riemannian metric and let \(\omega\) be a closed SD form. Viewing \(\omega\) as a section of \(\Lambda^2_+T^{*}_X\), assume it is transverse to the zero section; that is, for each \(x\in Z = \omega^{-1}(0)\), 
\(\nabla \omega_x\) has rank \(3\). Then \(Z\) is a smooth \(1\)-manifold, and \(\omega\) is necessarily near-symplectic: for each \(x\in Z\) and \(\xi\in T_X|_x\), the SD form \(\nabla_\xi\omega_x\) is nonzero; hence, its wedge 
square is also nonzero.
\smallskip%

Auroux, Donaldson, and Katzarkov proved a converse 
(Proposition 1 in \cite{A-D-K}):
\begin{theorem*}[\cite{A-D-K}]
Suppose \(\omega\) is a near-symplectic form on \(X\). 
Then there is a Riemannian metric \(g\) on \(X\) such that 
\(\omega\) is self-dual with respect to \(g\).
\end{theorem*}
\smallskip%

A smooth hypersurface \(\Sigma\subset X\) 
is called \(\omega\)-convex if there exists a vector field \(V\) defined near \(\Sigma\) such that
\(L_V\omega=\omega\), \(V\) is transverse to \(\Sigma\), and \(V\) points outward along \(\Sigma\) from the 
compact region it bounds. Such a vector field \(V\) is called an \(\omega\)-Liouville vector field. An \(\omega\)-convex surface \(\Sigma\) is necessarily transverse to \(Z\), and the primitive \(\lambda=\iota_V\omega\) restricts to a near-contact form on \(\Sigma\). The same applies to \(\omega\)-concave surfaces, obtained by replacing \(L_V\omega=\omega\) with \(L_V\omega=-\omega\).
\smallskip%

\begin{lemma}[Near-symplectic cobordism]\label{cobordism}
Let \(\lambda_t\), \(t\in[0,1]\), be a smooth 
family of near-contact forms on a closed 
\(3\)-manifold \(Y\). Then there exists a near-symplectic form \(\omega\) on \(W=[0,1]\times Y\) and \(\omega\)-Liouville vector fields \(V_0\) and \(V_1\), defined near \(\{0\}\times Y\) and \(\{1\}\times Y\), respectively, such that \(V_0\) is transverse to \(\{0\}\times Y\) and points into \(W\), \(V_1\) is transverse to \(\{1\}\times Y\) and points out of \(W\), and
\[
\iota_{V_0}\omega\big|_{\{0\}\times Y}=c_1\lambda_0,
\qquad
\iota_{V_1}\omega\big|_{\{1\}\times Y}=c_2\lambda_1,
\]
for some constants \(c_1,c_2>0\). In particular, \(\{0\}\times Y\) is \(\omega\)-concave and \(\{1\}\times Y\) is \(\omega\)-convex.
\end{lemma}
\begin{proof}
To begin with, assume that all 
\(\lambda_t\) vanish at the same finite set of points
\(p_1,\dots,p_k\in Y\), and that for each \(p_i\) the linearization \((\nabla\lambda_t)_{p_i}\) is independent of \(t\). We also assume the family is constant near the 
endpoints: there exists \(\varepsilon>0\) such that
\(\lambda_t=\lambda_0\) for \(t\in[0,\varepsilon]\) and
\(\lambda_t=\lambda_1\) for \(t\in[1-\varepsilon,1]\).
\smallskip%

We use the same parameter \(t\in[0,1]\) 
as the coordinate on the first factor of 
\(W=[0,1]\times Y\). 
For each \(t\), regard \(\lambda_t\) as a \(1\)-form on the slice \(\{t\}\times Y\). These assemble to a smooth \(1\)-form \(\lambda\) on \(W\), with no \(dt\)-component:
\[
\lambda(\partial_t)=0, \quad 
\lambda|_{\{t\}\times Y}=\lambda_t.
\]
We seek a function \(h=h(t)\) such that the \(2\)-form
\[
d(h\lambda)
\]
is near-symplectic on \(W\). Using the decomposition
\[
d = dt\wedge\partial_t + d_Y,
\]
we compute:
\[
d(h\lambda)
 = dh\wedge\lambda + h\,d\lambda
 = \dot h\,dt\wedge\lambda_t
   + h\,dt\wedge\dot\lambda_t
   + h\,d_Y\lambda_t,
\]
where \(\dot\lambda_t=\partial_t\lambda_t\). Thus:
\[
d(h\lambda)\wedge d(h\lambda)
 = 2h\,dt\wedge
   \Bigl(
      \dot h\,\lambda_t\wedge d_Y\lambda_t
      + h\,\dot\lambda_t\wedge d_Y\lambda_t
   \Bigr).
\]
Choose a volume form \(\operatorname{vol}\) on \(Y\) and 
a compatible connection \(\nabla\). Define smooth families of functions \(f_t, g_t \colon Y \to \mathbb{R}\) as:
\[
\lambda_t \wedge d_Y\lambda_t = f_t\,\operatorname{vol},
\quad
\dot\lambda_t \wedge d_Y\lambda_t = g_t\,\operatorname{vol}.
\]
Setting \(h(t)=e^{Ct}\) with \(C>0\), we find: 
\[
d(h\lambda)\wedge d(h\lambda)
 = 2 e^{2Ct}\, dt\wedge
   \bigl( C f_t + g_t \bigr)\,\operatorname{vol}.
\]
Let \(p_i\in Y\) be a zero of \(\lambda_t\). 
By Lemma \ref{lambda_closed}, \((d_Y\lambda_t)_{p_i}=0\). Since 
\((\nabla\lambda_t)_{p_i}\) is independent of \(t\), 
it follows that 
\((\nabla\dot\lambda_t)_{p_i}=0\). Thus, 
for each zero \(p_i\) and each \(t\), 
\[
g_t(p_i)=0,\quad (dg_t)_{p_i}=0,\quad (\nabla^2 g_t)_{p_i}=0.
\]
On the other hand, the near-contact condition on \(\lambda_t\) gives \(f_t(p_i)=0\), \((df_t)_{p_i}=0\), and the Hessian
\((\nabla^2 f_t)_{p_i}\) is positive-definite. 
Hence, \(C f_t + g_t\) has a nondegenerate minimum at 
each zero \(p_i\). It follows that there exist neighbourhoods \(U_i\ni p_i\) such that for each \(C > 1\) and each \(t\in[0,1]\), the function \(C f_t + g_t\) is positive on \(U_i - \{p_i\}\). Set \(U = U_1\cup\cdots\cup U_k\). 
Since \(g_t\) is uniformly bounded and \(Y - U\) 
is compact, it follows that for all large enough \(C > 0\), 
the function \(C f_t + g_t\) is positive on
\(Y-\{p_1,\dots,p_k\}\) for all \(t\in[0,1]\).
\smallskip%

Define a function \(f \colon W\to\mathbb{R}\) as:
\[
f(t,y)= e^{2Ct}\bigl(C f_t(y) + g_t(y)\bigr).
\]
Then \(f > 0\) away from the arcs \([0,1]\times\{p_i\}\), and each such arc is a Morse-Bott critical submanifold of \(f\). 
Hence, \(d(e^{Ct}\lambda)\) is near-symplectic on \(W\).
\smallskip%

Near \(t=1\) the family \(\lambda_t\) is constant, so
\(d(e^{Ct}\lambda_t)\) coincides 
with \(d(e^{Ct}\lambda_1)\) on a collar of
\(\{1\}\times Y\). Moreover,
\[
L_{\partial_t}\, d(e^{Ct}\lambda_t) = C\, d(e^{Ct}\lambda_t),
\]
so \(C^{-1}\partial_t\) is a Liouville vector field for \(d(e^{Ct}\lambda_t)\) near \(\partial W\). This completes the proof. \qed
\end{proof}
\smallskip%

Taubes \cite{Taub-7} and Perutz \cite{Perutz} show that a near-symplectic form \(\omega\) canonically orients its zero set \(Z=\omega^{-1}(0)\). For completeness, we recall their construction.
\smallskip

Let \(\omega\) be a near-symplectic form on \(X\), and let
\(Z=\omega^{-1}(0)\) be its \(1\)-dimensional zero set.
For each \(x\in Z\), the range 
of \(\nabla\omega_x\) is a \(3\)-plane
\(P_x\subset\Lambda^2T^{*}_{X}|_x\), on 
which the wedge pairing is positive-definite. 
Let \(N_x\subset T_X|_x\) be a complementary subspace to
\(T_Z|_x\). Then \(\nabla\omega_x : N_x \to P_x\) is an isomorphism. Let \(\partial_Z\in T_Z|_x\) be a 
tangent vector. Define the bilinear form:
\begin{equation}\label{A_omega}
A(\xi,\eta)
= \bigl( \iota_{\partial_Z}(\nabla_\xi\omega) \bigr)_x(\eta),
\quad \xi,\eta\in N_x.
\end{equation}
Since \(P_x\) is positive-definite, 
\(A\) is nondegenerate, and thus the sign of
\(\operatorname{det} A\) is defined. The canonical orientation on \(Z\) is defined by requiring that
\(\partial_Z\) be chosen so that 
\(\operatorname{det} A > 0\).
\smallskip%

\(d\omega=0\) and \((\nabla_{\partial_Z}\omega)_x=0\)
imply that \(A\) is symmetric.
\smallskip%

Recall an additional property of \(A\), found in 
\cite{Taub-7, Perutz}. Orient \(X\) so that \(\omega\wedge\omega\ge 0\). 
Let \(x\in Z\) and let \(P_x\) be the range of \(\nabla\omega_x\). The maximal positive-definite \(3\)-plane \(P_x\) determines a conformal
class of metrics \(g_x\) at \(T_X|_x\) such that 
\(P_x=\Lambda_{+}^{2} T^{*}_{X}|_x\). 
Choose \(N_x\subset T_X|_x\) to be the \(g_x\)-orthogonal complement of \(T_Z|_x\). With \(\partial_Z\) and \(A\) as above, the form \(A\) is
\(g_x\)-traceless: for each \(g_x\)-orthonormal basis
\(e_1,e_2,e_3\) of \(N_x\),
\[
\sum_{i=1}^3 A(e_i,e_i)=0.
\]
Let us now specialize to \(\mathbb{R}^4\) and 
consider near-symplectic 
forms whose coefficients are homogeneous of degree \(1\). 
Let \(\omega\) be such a form. 
Since \(\omega\) is homogeneous of degree \(1\) and near-symplectic, its zero set \(Z=\omega^{-1}(0)\) is a line in \(\mathbb{R}^4\). Let us classify these forms up to deformation; this is a reformulation of \cite{Taub-7, Perutz}.
\begin{lemma}\label{linear_connected}
Let \(\mathbb{R}^4\) be oriented and let \(\ell\subset\mathbb{R}^4\) be a line.
Consider the space of near-symplectic forms \(\omega\) on \(\mathbb{R}^4\)
with homogeneous coefficients of degree \(1\), such that
\(\omega\wedge\omega\ge 0\) and whose zero set is exactly \(\ell\).
Then this space has precisely two connected components, distinguished
by the orientation they induce on \(\ell\).
\end{lemma}
\begin{proof}
To begin with, 
choose a linear metric \(g\) on \(\mathbb{R}^4\) such that the range
of \(\nabla\omega_0\) consists of \(g\)-self-dual forms.
Since \(\omega\) has homogeneous linear coefficients, it 
follows that 
\(\omega_x = (\nabla_x\omega)_0\) for all 
\(x\in\mathbb{R}^4\); in particular, 
\(\omega\) is \(g\)-self-dual. The space of linear metrics is convex, so it suffices to classify
homogeneous linear self-dual forms \(\omega\) on \(\mathbb{R}^4\) whose zero
set is \(\ell\). We show that this space has two connected components.
\smallskip%

Let \(N\subset\mathbb{R}^4\) be the \(g\)-orthogonal complement of \(\ell\) at
the origin. Fix \(\partial_\ell\in\ell\), and let \(\nabla\) be the flat connection
associated to \(g\). As in \eqref{A_omega}, define a bilinear form:
\[
A(\xi,\eta)
= \bigl(\iota_{\partial_\ell}(\nabla_\xi\omega)\bigr)_0(\eta),
\qquad \xi,\eta\in N.
\]
Then \(A\) is nondegenerate, symmetric, and traceless.
For a homogeneous linear self-dual form \(\omega\) vanishing exactly on
\(\ell\), the form \(A\) determines \(\omega\) uniquely.  
Conversely, any nondegenerate symmetric traceless form \(A\) on \(N\) arises from such an \(\omega\).
\smallskip%

The space of nondegenerate symmetric traceless 
bilinear forms on \(N\) has two connected components, distinguished by the sign of
\(\operatorname{det} A\). This completes the proof. \qed
\end{proof}
\smallskip%

Let us now specialize to the case when \(X\) is a ball in \(\mathbb R^4\). 
\begin{proposition}\label{near-ball}
Let $B \subset \mathbb{R}^4$ be the closed unit ball with boundary
$\Sigma = \partial B$, and let $\omega$ and $\omega'$ be near-symplectic
forms defined in a neighbourhood of $B$ such that $\omega \wedge \omega$
and $\omega' \wedge \omega'$ induce the same orientation on $B$. Suppose
that $\Sigma$ is both $\omega$–convex and $\omega'$–convex, and let
$\lambda$ and $\lambda'$ be the induced near-contact forms on $\Sigma$.
Let
\[
Z = \omega^{-1}(0), \quad Z' = {\omega'}^{-1}(0).
\]
Assume that $Z$ and $Z'$ are unions of arcs (i.e.\ have no circular
components) and that $Z$ and $Z'$ have the same number of connected
components. Then $\lambda$ and $\lambda'$ are homotopic through
near-contact forms on $\Sigma$.
\smallskip

Moreover, suppose that the antipodal $\mathbb Z_2$–action on $B$ preserves
$\omega$ and $\omega'$, and that there exist $\mathbb Z_2$–invariant
$\omega$– and $\omega'$–Liouville vector fields near $\Sigma$ making
$\Sigma$ $\omega$– and $\omega'$–convex. Then the above homotopy may be
chosen $\mathbb Z_2$–equivariant, so the induced near-contact forms on the
quotient $\Sigma/\mathbb Z_2 \cong \mathbb{RP}^3$ are homotopic.
\end{proposition}
\begin{proof}
For simplicity, we treat only the non-equivariant statement; 
the antipodal case is identical, with all constructions 
carried out \(\mathbb Z_2\)-equivariantly.  
\smallskip%

Let $L$ be a closed manifold and define its cone:
\[
C(L) = \bigl([0,1]\times L\bigr)\big/\sim,\qquad 
(t_1,x_1)\sim (t_2,x_2)\ \Longleftrightarrow\ x_1=x_2,\ t_1=t_2=1.
\]
Let $p_{*}$ be the apex and 
$L_0 = \{0\}\times L \subset C(L)$.
\smallskip%

A link in $C(L)$ is a finite union of disjoint 
embedded arcs
$\gamma_i\subset C(L) - \{p_*\}$ with endpoints in $L_0$, each
transverse to $L_0$; an oriented link is such a link with an orientation
on each component. Two oriented links are 
isotopic with free endpoints in $L_0$ if they are connected by a smooth $1$–parameter
family of oriented links as above. It is standard that if $\operatorname{dim} L \ge 3$, 
any two oriented links in $C(L)$
are isotopic with free endpoints in $L_0$; this 
depends only on dimension of \(L\), not on \(\pi_1(L)\).
\smallskip%

Now strengthen this as follows. 
Let $D_1,\dots,D_k\subset L_0$ be
pairwise disjoint closed $3$–discs. Each cone \(C(D_i)\) embeds in \(C(L)\), as does the wedge
\[
C(D_1)\vee\cdots\vee C(D_k)\subset C(L),
\]
which is a strong deformation retract of \(C(L)\). Let \(Z = \gamma_1\cup \cdots \cup\gamma_k \subset C(L)\) be 
a link. We may isotope $Z$ so that 
there are disjoint $D_i\subset L_0$ such that 
$\gamma_i\subset C(D_i)$. Then $C(L) - Z$ deformation retracts onto the wedge
\[
\bigl(C(D_1) - \gamma_1\bigr)\ \vee \cdots \vee\
\bigl(C(D_k) - \gamma_k\bigr).
\]
Each $C(D_i)$ is a pointed $4$–ball; therefore, 
\(C(D_i) - \gamma_i\) retracts onto a pointed \(2\)-sphere. 
Hence, $C(L)- Z$ deformation retracts (relative to $p_*$) onto a wedge of $k$ 2-spheres. Again, only the 
dimension of \(L\) is used.
\smallskip%

Let us apply this with \(L = \Sigma = \partial B\), so 
that \(C(L)\) is \(B\), and 
\(p_*\) is the origin of \(B\). Let $\omega$ and $\omega'$ be near-symplectic 
forms on $B$ with zero sets
\[
Z = \omega^{-1}(0),\quad Z' = {\omega'}^{-1}(0).
\]
Both \(Z\) and \(Z'\) are oriented links in \(B\) and 
have the same number of components. By the above link isotopy
statement, there is an isotopy 
of \(Z'\) to \(Z\). Thus, we may assume that 
\(Z' = Z\) and the canonical orientations on
$Z$ induced by $\omega$ and $\omega'$ coincide.
\smallskip%

Set:
\[
Z\cap\Sigma = \{p_1,\dots,p_k\},\quad \Sigma = \partial B.
\]
Let $V$ and $V'$ be Liouville vector fields for $\omega$ and $\omega'$, defined in tubular neighbourhoods of $\Sigma$ and transverse to
$\Sigma$. Set:
\[
\lambda = \iota_V\omega,\quad \lambda' = \iota_{V'}\omega'.
\]
Since $\omega\wedge\omega$ and $\omega'\wedge\omega'$ 
are cooriented in \(B\), the $3$–forms $\lambda\wedge d\lambda$ and
$\lambda'\wedge d\lambda'$ are cooriented on $\Sigma$. We shall show
that $\lambda$ and $\lambda'$ are homotopic as 
non-vanishing $1$–forms 
on $\Sigma - \{p_1,\dots,p_k\}$. By
Lemma \ref{contact_classification}, this yields a homotopy through near-contact forms.
\smallskip%

In fact, we construct a homotopy between $\omega$ and $\omega'$ through nondegenerate $2$–forms on $B-Z$; this induces 
the desired homotopy between \(\lambda\) and \(\lambda'\).
\smallskip%

Consider the bundle $\mathcal{N} \to B$ of nondegenerate 
\(2\)-forms on \(B\). It has two connected components, distinguished by the orientation induced on \(B\). Let 
\(\mathcal{N}^{+} \subset \mathcal{N}\) be the connected 
component of those forms which induce the same orientation as 
\(\omega\). Since \(\omega\) and \(\omega'\) are nondegenerate and compatible with the orientation, they define sections \(\omega,\omega' \colon B-Z \to \mathcal{N}^{+}|_{B-Z}\). 
\smallskip%

$B-Z$ deformation retracts onto a wedge of pieces
$C(D_i)-\gamma_i$, and each $C(D_i)-\gamma_i$ retracts onto a small $2$–sphere $S_i$ linking $\gamma_i$. Thus, it suffices to show that $\omega$ and $\omega'$ are homotopic as sections of $\mathcal{N}^{+} \to S_i$ for each $S_i$. 
\smallskip%

Fix \(i\). Let \(U_i\) be an arbitrary small neighbourhood 
of an interior point of \(\gamma_i\). By a small 
\(C^1\)-perturbation, arrange that \(\omega\) and 
\(\omega'\) coincide with their linear parts 
in \(U_i\). Deform \(S_i\) so 
that \(S_i \subset U_i\). Since \(\omega\) and 
\(\omega'\) induce the same orientation 
on \(\gamma_i\), it follows from 
Lemma \ref{linear_connected} that there is 
a near-symplectic deformation of \(\omega\) 
to \(\omega'\) on \(U_i - \gamma_i\). This completes 
the proof. \qed
\end{proof}
\smallskip%

Proposition \ref{near-ball} and 
Lemma \ref{cobordism} allow us to introduce the 
following near-symplectic surgery. Let 
$B \subset \mathbb{R}^4$ be the unit ball with 
boundary $\Sigma = \partial B$. Let 
$\omega, \omega'$ be near-symplectic forms 
defined in a neighbourhood of $B$ such that 
$\omega$ and $\omega'$ induce the same orientation on 
$B$. Suppose that $\Sigma$ is both 
$\omega$– and $\omega'$–convex, and that 
both $Z = \omega^{-1}(0)$ and $Z' = {\omega'}^{-1}(0)$ 
are unions of arcs and have the same number of connected 
components. Then there exist a smaller ball 
$B_0 \subset B$ and a near-symplectic form 
$\hat{\omega}$ defined in a neighbourhood of $B$ 
such that $\hat{\omega} = \omega$ near 
$\partial B$ and such that 
$(B_0, \hat{\omega})$ is diffeomorphic to 
$(B, c\,\omega')$ for some constant $c > 0$.
\smallskip

Moreover, suppose that the antipodal 
$\mathbb{Z}_2$–action on $B$ preserves $\omega$ and $\omega'$, and there exist $\mathbb{Z}_2$–invariant $\omega$– and $\omega'$–Liouville vector fields near $\Sigma$ making $\Sigma$ $\omega$– and $\omega'$–convex. Then 
$\hat{\omega}$ can be chosen 
$\mathbb{Z}_2$–invariant.

\section{Proof of Proposition \ref{cone_zero}}\label{assemble-section} 
It suffices to treat the case \(n=1\); the general case is no harder. Let 
\((M,g)\) be an oriented Riemannian \(4\)-orbifold with a double point \(p\). Choose a neighbourhood \(V\) of \(p\) equipped with an isometry
\(V \simeq (B/\{\pm1\}, g_0)\), 
where \(B \subset \mathbb R^4\) is a ball centered 
at the origin and \(g_0\) is the flat metric. Let \(\psi\) be a closed SD form on \(M\). We may assume \(\psi(p)\neq 0\), otherwise there is nothing to prove. By Darboux’ theorem, \((V,\psi)\) is diffeomorphic to the quotient 
\((B,\omega_0)/\{\pm1\}\), where \(\omega_0\) is the standard symplectic form on \(B\).
\smallskip%

Proposition~\ref{cone_zero} follows from the following 
stronger statement. 
\begin{proposition}\label{cone_zero_ball}
Let \(B \subset \mathbb R^4\) be 
the unit ball centered at the origin, and let \(g_0\) be the flat metric. Then there exists a near-symplectic form \(\omega\) in a neighbourhood of \(B\) such that 
\(\omega\) vanishes at \(0 \in B^4\), \(\omega\) is invariant 
under the antipodal map, \(\omega\) coincides 
with the standard 
symplectic form in a neighbourhood of \(\partial B\), and 
\(\omega\) is \(g_0\)-self-dual in a neighbourhood 
of the origin.
\end{proposition}
\smallskip%

If this is the case, 
replace \(\psi\) on \(V\) by such an \(\omega\). Then the theorem of Auroux-Donaldson-Katzarkov applies to keep \(\psi\) self-dual for some metric that agrees with \(g_0\) near \(p\). The rest of this section is the proof of Proposition~\ref{cone_zero_ball}.
\smallskip%

\begin{lemma}
Let $F = (f_1,f_2,f_3) \colon 
\mathbb{R}^4 \to \mathbb{R}^3$ be defined as:
\[
\begin{aligned}
f_1(x) &= 2x_1x_3 - 2x_2x_4 - x_3x_4,\\
f_2(x) &= -x_1(4x_2 + x_3),\\
f_3(x) &= x_2^2 + x_3^2 - H(x_1,x_4),
\end{aligned}
\qquad
H(x_1,x_4) = 3x_1^2 - x_1x_4 - x_4^2.
\]
Then for each $\varepsilon \ge 0$,
\[
F^{-1}(0,0,-\varepsilon)
= \{x \in \mathbb{R}^4 : x_2 = x_3 = 0,\ H(x_1,x_4) = \varepsilon\}.
\]
For $\varepsilon>0$, $dF$ has full rank 
on $F^{-1}(0,0,-\varepsilon)$, and for $\varepsilon=0$, it has full rank on $F^{-1}(0,0,0)-\{0\}$.
\smallskip%

Set:
\begin{equation}\label{omega_f}
\omega = f_1\,\omega_1 + f_2\,\omega_2 + f_3\,\omega_3,
\end{equation}
where
\[
\omega_1 = dx_1\wedge dx_2 + dx_3\wedge dx_4,\quad
\omega_2 = dx_1\wedge dx_3 - dx_2\wedge dx_4,\quad
\omega_3 = dx_1\wedge dx_4 + dx_2\wedge dx_3.
\]
Then \(\omega\) is closed and self-dual with respect to the standard flat metric and orientation. In particular, \(\omega\) is near-symplectic on \(\mathbb{R}^4 - \{0\}\), and \(\omega+\varepsilon\omega_3\) is near-symplectic on \(\mathbb{R}^4\) for all \(\varepsilon>0\).
\end{lemma}
\begin{proof}
All statements follow by a direct computation. \qed
\end{proof}
\smallskip%

\begin{lemma}\label{cross_perturb}
Let \(\omega\) be as in \eqref{omega_f}. Then there 
exists a small, antipodal-invariant 
perturbation \(\omega_{\varepsilon}\) of \(\omega\) 
such that \(\omega_{\varepsilon} = \omega\) outside a neighbourhood of the origin, \(\omega_{\varepsilon}\) is near-symplectic on all of \(\mathbb R^4\), and its zero 
set \(\omega^{-1}_{\varepsilon}(0)\) has exactly two non-compact connected components.
\end{lemma}
\begin{proof}
Let 
\[
E = \sum_{i = 1}^{4} x_i \partial_{x_i}
\]
be the Euler field. Let 
\(\omega_3\) be as in \eqref{omega_f}. Set:
\[
\lambda = \frac{1}{4}\,\iota_E \omega,\quad 
\mu = \frac{1}{2}\,\iota_E \omega_3,
\]
so that \(d\lambda = \omega\) and \(d\mu = \omega_3\).
\smallskip%

Let \(\rho \colon \mathbb{R}^4 \to \mathbb{R}\) 
be a radial cut-off function such that 
\(\rho(x) = \rho(-x)\), \(\rho = 1\) on a ball \(U\) about the origin, and \(\rho = 0\) outside a slightly 
larger ball. Set:
\[
\omega_{\varepsilon} = d(\lambda + \varepsilon \rho \mu).
\]
Since \(\lambda,\mu,\rho\) are invariant 
under \(x \mapsto -x\), so is 
\(\omega_{\varepsilon}\). On the region \(\rho=1\) (in particular, on \(U\)), 
\[
\omega_{\varepsilon} = \omega + \varepsilon\omega_3,
\]
which is near-symplectic for all \(\varepsilon > 0\). On \(\rho=0\), \(\omega_{\varepsilon} = \omega\), which is also 
near-symplectic. For all small \(\varepsilon\), \(\omega_{\varepsilon}\) remains near-symplectic on the support of \(d\rho\). Hence, \(\omega_{\varepsilon}\) is near-symplectic on all of \(\mathbb R^4\). 
\smallskip%

$\omega^{-1}(0)$ is a union of two lines through the origin, and the zero set of $\omega + \varepsilon\omega_3$ is a small smoothing of this cross into two branches of the hyperbola $H(x_1,x_4)=\varepsilon$. Hence, for all small \(\varepsilon > 0\), $\omega_{\varepsilon}^{-1}(0)\cap U$ consists exactly of \(2\) arcs, each with two ends 
on \(\partial U\). Let \(K\) be a closed 
neighbourhood of \(\mathbb R^4-U\), strictly larger than \(\mathbb R^4-U\). Then \(\omega_{\varepsilon}-\omega\) has compact support in \(K\) and is \(C^1\)-small. 
Hence, for all small \(\varepsilon > 0\), $\omega^{-1}_{\varepsilon}(0) \cap (\mathbb R^4 - U)$ consists exactly of \(4\) rays. These rays attach to the two arcs to produce two connected components. 
This completes the proof. \qed
\end{proof}
\medskip%

\emph{Proof of Proposition~\ref{cone_zero_ball}.} 
Consider \(\mathbb R^4\) with the standard symplectic form \(\omega_{\mathrm{st}}\) and the unit ball 
\(B \subset \mathbb R^4\). Let us 
construct the desired near-symplectic form by 
a sequence of modifications 
of \(\omega_{\mathrm{st}}\). The first step uses the following result of Taubes. See \cite{Taub-5}. 
\begin{theorem*}[\cite{Taub-5}]
There exists a near-symplectic form 
\(\omega\) on \(\mathbb R^4\) such that its zero set \(\omega^{-1}(0)\) is a circle near the origin 
and \(\omega\) coincides with a standard symplectic form outside a neighbourhood of the origin.
\end{theorem*}
\smallskip%

Let us apply Taubes' construction 
away from the origin. Pick a point 
\(x \in \mathbb R^4 - \{0\}\) and a 
small neighbourhood \(U\ni x\). Taubes' argument allows us to modify \(\omega_{\mathrm{st}}\) inside \(U\) to create a single zero circle in \(U\). To keep antipodal invariance, perform the same modification near \(-x\). Denote the resulting form by \(\omega\).
\smallskip%

\(Z = \omega^{-1}(0)\) is a circle 
\(C \subset U\) and its antipodal image. 
Pick a point \(p \in C\). Perform a small \(C^1\)-perturbation of \(\omega\) near \(p\) so that \(\omega\) coincides with its linear part there. This immediately 
yields a ball neighbourhood \(D \ni p\) whose boundary \(\Sigma = \partial D\) is \(\omega\)-convex. To keep antipodal invariance, perform the same perturbation near \(-p\).
\smallskip%

Choose a ball neighbourhood \(D_0\) of 
the origin \(0 \in B\) with \(\omega\)-convex boundary \(\Sigma_0 = \partial D_0\). Pick a 
point \(a \in \Sigma_0\) and a point \(b \in \Sigma\) with \(\omega_{b} \neq 0\). Join \(a\) to \(b\) by a path \(\delta\). To keep antipodal invariance, also introduce the antipodal points \(-a,-b\) and path \(-\delta\).
\smallskip%

Form the \(\omega\)-convex connected sum of 
\(D, D_0, (-D)\) along paths \(\delta\) and \(-\delta\). That is, choose an antipodal-invariant, \(\omega\)-convex neighbourhood
\[
X \supset D \cup D_0 \cup (-D).
\]
\((X,\omega)\) is a near-symplectic \(4\)-ball with convex boundary embedded in \((B,\omega)\), and the embedding is antipodal-invariant. \(\omega^{-1}(0)\cap X\) consists 
of \(2\) arcs and no circular components.
\smallskip%

Since \(X\) is a ball, it admits a near-symplectic form \(\omega_{\varepsilon}\) as in Lemma~\ref{cross_perturb}. 
By construction, \(\partial X\) is \(\omega_{\varepsilon}\)-convex. Both \(\omega|_X\) and \(\omega_{\varepsilon}\) are antipodal-invariant and may be assumed cooriented. Then, as explained at the end of \S\,\ref{near-symplectic-section}, one can modify \(\omega\) only near \(X \subset B\) so that 
\[
\omega = c\,\omega_{\varepsilon} \quad \text{on } X
\]
for some constant \(c>0\), and \(\omega\) remains antipodal-invariant.
\smallskip%

Finally, as in Lemma~\ref{cross_perturb}, deform \(c \omega_{\varepsilon}\) to a homogeneous SD form (with respect to a flat metric) that vanishes at the 
fixed point of the antipodal map. The resulting form \(\omega\) on \(B\) satisfies all properties required in Proposition~\ref{cone_zero_ball}, and 
the proof follows.
\bibliographystyle{plain}
\bibliography{references}

\end{document}